\documentclass{amsart}
\usepackage[colorlinks,linkcolor=black,citecolor=black,urlcolor=black, linktocpage=true,dvips]{hyperref}
\usepackage{srcltx}
\usepackage[english]{babel}
\numberwithin{equation}{section}

\title[Pseudo-differential Operators on Fractals]{Pseudo-differential Operators on Fractals and Other 
Metric Measure Spaces}

\author[M. Ionescu and L. G. Rogers and R. S. Strichartz]{Marius
  Ionescu and Luke G. Rogers and Robert S. Strichartz} 

\address[mionescu@colgate.edu]{{\sc Marius Ionescu}: Department of Mathematics, Colgate
  University, NY, 13346, USA}
\address[rogers@math.uconn.edu]{{\sc Luke G. Rogers}: Department of Mathematics, University
  of Connecticut, Storrs, CT,06269, USA}
\address[str@math.cornell.edu]{{\sc Robert S. Strichartz}: Department of Mathematics, Cornell University, Ithaca, NY, 14853,
USA}

\thanks{The work of the first author was partially supported by a grant from
the Simons Foundation (\#209277 to Marius Ionescu).}

\thanks{The work of the third author was partially supported by the National
Science Foundation, grant DMS-0652440.}



 \usepackage{amstext}
 \usepackage{amssymb}

  \theoremstyle{remark}
  \newtheorem*{acknowledgement*}{\protect\acknowledgementname}
\theoremstyle{plain}
\newtheorem{thm}{\protect\theoremname}[section]
  \theoremstyle{definition}
  \newtheorem{defn}[thm]{\protect\definitionname}
  \theoremstyle{remark}
  \newtheorem{rem}[thm]{\protect\remarkname}
  \theoremstyle{plain}
  \newtheorem{prop}[thm]{\protect\propositionname}
  \theoremstyle{plain}
  \newtheorem{lem}[thm]{\protect\lemmaname}
  \theoremstyle{plain}
  \newtheorem{cor}[thm]{\protect\corollaryname}
  \theoremstyle{definition}
  \newtheorem{example}[thm]{\protect\examplename}

  \providecommand{\acknowledgementname}{Acknowledgement}
  \providecommand{\corollaryname}{Corollary}
  \providecommand{\definitionname}{Definition}
  \providecommand{\examplename}{Example}
  \providecommand{\lemmaname}{Lemma}
  \providecommand{\propositionname}{Proposition}
  \providecommand{\remarkname}{Remark}
\providecommand{\theoremname}{Theorem}

\DeclareMathOperator{\singsupp}{sing~supp}

\begin{document}
\maketitle


%
%

\begin{abstract}
We define and study pseudo-differential operators on a class of fractals
that include the post-critically finite self-similar sets and Sierpinski
carpets. Using the sub-Gaussian estimates of the heat operator we
prove that our operators have kernels that decay and, in the constant
coefficient case, are smooth off the diagonal. Our analysis can be
extended to products of fractals. While our results are applicable
to a larger class of metric measure spaces with  Laplacian, we use
them to study elliptic, hypoelliptic, and quasi-elliptic operators
on p.c.f. fractals, answering a few open questions posed in a series
of recent papers. We extend our class of operators to include the
so called H\"ormander hypoelliptic operators and we initiate the
study of wavefront sets and microlocal analysis on p.c.f. fractals.

\end{abstract}

%
%

\section{Introduction}

In this paper we define and study pseudo-differential operators on
metric measure spaces endowed with a non-positive self-adjoint Laplacian
such that the heat operator satisfies sub-Gaussian estimates.
The motivating examples for our work are metric measure spaces
constructed from p.c.f. fractals
(\cite{Kig_CUP01,Str:TAM04,Str_CJM98}). However, our results capture
the known facts from classical harmonic analysis on $\mathbf{R}^n$ and
Riemannian manifolds and can be
used to define pseudo-differential operators on other metric measure
spaces. For example, the operators defined in \cite{Heb_StudMath89}  fall
under the hypothesis of our main theorems.

Analysis on fractals from either the probabilistic or the analytic
viewpoint has been the focus of intense study recently (see, for example,
\cite{HaKu_PTRD03,BaPe_PTRF88,Kig_CUP01,Str_Prin06} and references
therein). Several recent papers have studied properties of spectral operators
on fractals. For example, \cite{Str:MRL05} shows some new convergence
properties of Fourier series on fractals with spectral gaps and establishes a  Littlewood-Paley inequality for such fractals.
In \cite{IoPeRoHuSt_08} a kernel formula for the resolvent of
the Laplacian on any p.c.f. fractal is given; this
generalizes Kigami's result for the Green function on p.c.f. fractals
(\cite{Kig_CUP01}). These results have been extended to infinite
blowups of fractals in \cite{Rog_08}. Numerical results suggest
that this result might hold for other spectral operators on fractals
\cite{allan2009spectral,constantin2009analysis}. The main result
of \cite{IR_2010} says that if a spectral operator on a p.c.f. fractal
is given by integration with respect to a kernel that is smooth and
has a specific type of decay off the diagonal, then it is a Calder\'on-Zygmund
operator in the sense of \cite[Section I.6.5]{Ste93}. In particular,
the authors show that the Riesz and Bessel potentials and, more generally,
the Laplace type operators are Calder\'on-Zygmund operators. On the
other hand, in \cite{DuOuSi_JFA02,sikora-2008}, the authors show
that spectral multiplier operators on metric measure spaces and products
of such spaces are bounded on $L^{q}$ as long as the heat operator
satisfies specific estimates. The Laplacian defined on some p.c.f.
fractals and some highly symmetric Sierpinski gaskets satisfies these assumptions.

We begin by presenting the standing hypothesis of our results in
Section \ref{sec:background}. We also briefly review the definition
and main properties of p.c.f. fractals and the Laplacians defined on
them.

The pseudo-differential operators with constant coefficients
that are studied in Section \ref{sec:Symbols-and-pseudo-differential}
generalize the class of spectral multipliers studied in the papers
mentioned above. We prove that these operators satisfy the symbolic
calculus and that they are given by integration with respect to kernels
that are smooth and decay off the diagonal, extending some of the
results of \cite{IR_2010}. Moreover, the class of pseudo-differential
operators of order $0$ on p.c.f. fractals are Calder\'on-Zygmund
operators and thus extend to bounded operators on $L^{q}$,
for all $1<q<\infty$. In this context we therefore recover the results of \cite{DuOuSi_JFA02,sikora-2008}.
We extend our analysis to products of such spaces in Section \ref{sec:products}.

We define Sobolev spaces on spaces built out of p.c.f. fractals in Section \ref{sec:Sobolev-Spaces}
and prove that pseudo-differential operators with constant coefficients
are bounded on them. We study elliptic and hypoelliptic operators on fractals
in Section \ref{sec:Elliptic-and-Hypoelliptic}. Namely, we prove
that a pseudo-differential operator satisfies the pseudo-local properties
and that an elliptic operator is hypoelliptic.  This gives positive answers to some open questions posed in \cite{Str_Prin06,BoStr_IUMJ07,RogStr-2009_distr}.
An interesting class of operators that can be defined on fractals
with spectral gaps are the so called quasielliptic operators (\cite{BoStr_IUMJ07,sikora-2008}).
We show that every quasielliptic operator is equal to an elliptic
pseudo-differential operator, though there are quasielliptic
differential operators which are not elliptic as differential
operators. 

In Section \ref{sec:rho_Hormander} we extend the class of pseudo-differential
operators to include operators for which the derivatives of the symbols
have a slower rate of decay. As an application we show that the H\"ormander
type hypoelliptic operators are hypoelliptic. This extends one
side of the classical result (\cite[Chapter III, Theorem 2.1]{Tay_PSDO81}).
The converse is false in general, as is exemplified by the quasielliptic
operators (see Subsection \ref{sub:Quasielliptic-operators}).

Section \ref{sec:wavefront} introduces the wavefront set and microlocal
analysis on products of compact spaces built out of fractals. We show
that pseudo-differential operators may decrease the wavefront set
and that elliptic operators preserve the wavefront set, extending
results from classical harmonic analysis (see, for example,
\cite{sogge1993fourier})\textbf{.} We also describe the wavefront set
for a few specific examples. 

In the last section, we study properties of pseudo-differential
operators with variable coefficients on compact fractafolds. The main
results of the sections are the continuity off the diagonal of the kernels of such
operators  and the $L^{q}$-boundedness of these operators.
These results cannot be obtained using the methods of \cite{DuOuSi_JFA02}
and \cite{sikora-2008}. We also conjecture that the pseudo-differential
operators with variable coefficients of order $0$ are Calder\'on-Zygmund
operators.

Some of the results proved here are extensions of the corresponding
results from classical harmonic analysis (see, for example \cite{Tay_PSDO81,Ste_PMS30_70,Ste_PMS43_93})
However the proofs of our results are very different. The main reason for this difference
is that the product of smooth functions is, in general, no longer
in the domain of the Laplacian \cite{BassStrTep_JFA99}. Therefore
techniques that are essential in real analysis like multiplication
with a smooth bump are not available to us. We frequently use the
Borel type theorem proved in \cite{RoStrTe_08} to decompose a smooth
function in a sum of smooth functions.
\begin{acknowledgement*}
The authors would like to thank Camil Muscalu and Alexander Teplyaev
for many useful conversations during the preparation of the
manuscript.
They also thank the referee for various suggestions, including
pointing out the relevance of \cite{Heb_StudMath89} to this topic. The
authors express their thanks to Yin-Tat Lee for pointing out an error
in an earlier draft of this article.
\end{acknowledgement*}

\section{\label{sec:background}Background}

\global\long\def\dom{\operatorname{dom}}
\global\long\def\E{\mathcal{E}}
 For the main results of Sections \ref{sec:Symbols-and-pseudo-differential}
and \ref{sec:products} we need a metric space $(X,R)$ with a Borel
measure $\mu$ and a negative self-adjoint Laplacian $\Delta$. We
assume that $X$ satisfies the doubling condition, that is, there
is a constant $C>0$ such that
\begin{equation}
  \label{eq:double}
\mu(B(x,2r))\le C\mu(B(x,r))\;\text{for all}\; x\in X\text{ and }r>0.  
\end{equation}

We also assume that the heat operator $e^{t\Delta}$ has a positive
kernel $h_{t}(x,y)$ that satisfies the following sub-Gaussian upper
estimate
\begin{equation}
h_{t}(x,y)\le c_{1}t^{-\beta}\exp\left(-c_{2}\left(\frac{R(x,y)^{d+1}}{t}\right)^{\gamma}\right),\label{eq:heat_estimates}
\end{equation}
where $c_{1},c_{2}>0$ are constants independent of $t,x$ and $y$.  In this expression both
$d$ and $\gamma$ are constants that depend on $X$ and $\beta=d/(d+1)$.
Moreover we assume that $h_{z}(x,y)$ is an holomorphic function on
$\{\operatorname{Re}z>0\}$.

Examples of spaces that satisfy the above hypothesis include the so
called p.c.f.~fractals and the highly symmetric Sierpinksi carpets.
We review  some of the definitions and properties of the p.c.f.
fractals and of the Laplacians defined on them. For more details the
reader can consult the books \cite{Kig_CUP01} and \cite{Str_Prin06}.
Recall that an \emph{iterated function system (i.f.s.)} is a collection
$\{F_{1},\dots,F_{N}\}$ of contractions on $\mathbb{R}^{d}$. For
such an i.f.s. there exists a unique self-similar set $K$ satisfying
(see \cite{Hut_81})
\[
K=F_{1}(K)\bigcup\dots\bigcup F_{N}(K).
\]
For $\omega_{1},\omega_{2},\dots\omega_{n}\in\{1,\dots,N\}$, $\omega=\omega_{1}\omega_{1}\dots\omega_{n}$
is a word of length $n$ over the alphabet $\{1,\dots,N\}$. The subset
$K_{\omega}=F_{\omega}(K):=F_{\omega_{1}}\circ\dots\circ F_{\omega_{n}}(K)$
is called a cell of level $n$. The set of all finite words over $\{1,\dots,N\}$
is denoted by $W_{*}$. Each map $F_{i}$ of the i.f.s. defining $K$
has a unique fixed point $x_{i}$. We say that $K$ is a \emph{post-critically
finite (p.c.f.) self-similar} set if there is a subset $V_{0}\subseteq\{x_{1},\dots,x_{N}\}$
satisfying

\[
F_{\omega}(K)\bigcap F_{\omega^{\prime}}(K)\subseteq F_{\omega}(V_{0})\bigcap F_{\omega^{\prime}}(V_{0})
\]
for any $\omega\ne\omega^{\prime}$ having the same length. The set
$V_{0}$ is called the \emph{boundary} of $K$ and the boundary of
a cell $K_{\omega}$ is $F_{\omega}(V_{0})$. One defines $V_{1}=\bigcup_{i}F_{i}(V_{0})$,
and, inductively, $V_{n}=\bigcup_{i}F_{i}(V_{n-1})$ for $n\ge2$.
The fractal $K$ is the closure of $\bigcup_{n}V_{n}$.

The Laplacian on p.c.f. fractals may be built using Kigami's construction
\cite{Kig_CUP01} from a self-similar Dirichlet energy form $\mathcal{E}$
on $K$ with weights $\{r_{1},\dots,r_{N}\}$:
\[
\mathcal{E}(u)=\sum_{i=1}^{N}r_{i}^{-1}\mathcal{E}(u\circ F_{i}).
\]
The existence of such forms is non-trivial, but on a large collection
of examples they may be obtained from the approximating graphs as
the appropriate renormalized limit of graph energies (\cite{Kig_CUP01,Str_Prin06}).

The second ingredient is the existence of a unique self-similar measure
\[
\mu(A)=\sum_{i=1}^{N}\mu_{i}\mu(F^{-1}(A)),
\]
 where $\{\mu_{1},\dots,\mu_{N}\}$ are weights such that $0<\mu_{i}<1$
and $\sum\mu_{i}=1$, see \cite{Hut_81}.

Then the Laplacian is defined weakly: $u\in\dom\Delta_{\mu}$ with
$\Delta_{\mu}u=f$ if
\[
\mathcal{E}(u,v)=-\int_{X}fvd\mu
\]
 for all $v\in\dom\mathcal{E}$ with $v|_{V_{0}}=0$. The domain of
the Laplacian depends on the assumptions that one makes about $f$.
Kigami (\cite{Kig_CUP01}) assumes that $f$ is continuous, but in this paper it will be more natural to assume that $f$ is in $L^{2}(\mu)$, which
gives a Sobolev space (see Section \ref{sec:Sobolev-Spaces}). We write $u\in\dom_{L^{p}}\Delta$ if $f\in L^{p}(\mu)$.

The \emph{effective resistance metric} $R(x,y)$ on $K$ is defined
via
\[
R(x,y)^{-1}=\min\{\mathcal{E}(u)\,:\, u(x)=0\mbox{ and }u(x)=1\}.
\]
 It is known that the resistance metric is topologically equivalent,
but not metrically equivalent to the Euclidean metric (\cite{Kig_CUP01,Str_Prin06}).

The unit interval and the Sierpinski gasket (\cite{Kig_CUP01,Str_Prin06,FuSh_92,BassStrTep_JFA99,NeStTe_04,Str_TAMS03,Tep_JFA98})
are important examples of p.c.f. fractals. Other examples are the
affine nested fractals \cite{FiHaKu_94} that in turn are generalizations
of the nested fractals \cite{Lin_MAMS90}.

Some of the spaces that we consider in this paper are built from p.c.f.
fractals as in \cite{Str_CJM98,Str_TAMS03}. In those papers the author
defines fractal blow-ups of a p.c.f. fractal $K$ and fractafolds
based on $K$. The former generalizes the relationship between the
unit interval and the real line to arbitrary p.c.f. self-similar sets,
while the latter is the natural analogue of a manifold. Let $w\in\{1,\dots,N\}^{\infty}$
be an infinite word. Then
\[
F_{w_{1}}^{-1}\dots F_{w_{m}}^{-1}K\subseteq F_{w_{1}}^{-1}\dots F_{w_{m}}^{-1}F_{w_{m+1}}^{-1}K.
\]
The \emph{fractal blow-up} is
\[
X=\bigcup_{m=1}^{\infty}F_{w_{1}}^{-1}\dots F_{w_{m}}^{-1}K.
\]
If $C$ is an $n$ cell in $K$, then $F_{w_{1}}^{-1}\dots F_{w_{m}}^{-1}C$
is called an $(n-m)$ cell. The blow-up depends on the choice of the
infinite word $w$. In general there are an uncountably infinite number
of blow-ups which are not homeomorphic. In this paper we assume that
the infinite blow-up $X$ has no boundary. This happens unless all
but a finite number of letters in $w$ are the same. One can extend
the definition of the energy $\E$ and measure $\mu$ to $X$. The
measure $\mu$ will be $\sigma$-finite rather than finite. Then one
can define a Laplacian on $X$ by the weak formulation. It is known
that, for a large class of p.c.f. fractals, the Laplacian on an infinite
blow-up without boundary has pure point spectrum (\cite{Tep_JFA98},
\cite{Sab:JFA00}).

One can view a fractal blow-up as a collection of copies of $K$ that
are glued together at boundary points. Every point of $X$ has a neighborhood
which is homeomorphic to a neighborhood of a point in $K$. A \emph{fractafold
\cite{Str_TAMS03}} based on $K$ is defined to be a set which satisfies
this latter property. We will work on a restricted class of fractafolds,
which may be thought of as more like triangulated manifolds. Specifically
we will consider fractafolds $X$ that consist of a finite or infinite
union of copies of $K$ glued together at some of the boundary points.
The fractafold $X$ is compact if and only if we consider a finite
number of copies of $K$. We suppose in the following that all the
copies of $K$ have the same size in $X$. If all the boundary points
of the copies of $K$ are paired, then the fractafold $X$ has no
boundary. When $K$ is the unit interval this construction produces
the unit circle. The next simplest example is the double cover of the Sierpinski
gasket, where one consider two copies of the fractal with corresponding
boundary points paired. One can extend the definition of energy and
Laplacian from $K$ to a fractafold based on $K$. An explicit
description of the spectral resolution of the fractafold Laplacian
for certain infinite fractafolds is given in \cite{Str_Tep_2010_arx}.

Products of fractals provide another important class of examples for
our results. An important point to keep in mind is that product of
p.c.f. fractals is not a p.c.f. fractal. Strichartz described
in \cite{Str:TAM04} how one can extend the definition of the Laplacian
and energy to products of fractals.

The estimates \eqref{eq:heat_estimates} are known to be true for
$0<t<1$ on
a large number of p.c.f. fractals (\cite{FiHaKu_94}, \cite{BaPe_PTRF88},
\cite{HaKu_PLMS99}, \cite{HaKu_PTRD03}, \cite{Str_08Quant}) and
Sierpinski carpets (\cite{BarBas_89,BarBas_CJM99}), with $R$ being
the resistance metric, the constant $d$ being the Hausdorff dimension
with respect to the resistance metric, and $\gamma$ a constant specific
to the fractal. For the Dirichlet Laplacian on a p.c.f. fractal the
estimate for $t>1$ is immediate because the heat kernel decays
exponentially at a rate 
determined by the smallest eigenvalue. In the case of the Neumann
Laplacian one should instead modify the heat kernel by subtracting its
projection 
onto the zero eigenspace of constant functions, which has exponential
decay controlled by the first non-zero eigenvalue. This latter 
change does not affect the definition of pseudo-differential operators
because their action is on the eigenspaces with non-zero eigenvalues; 
accordingly we will abuse notation and refer to this modified heat
kernel as just the heat kernel.  No modification is needed in the case
of blowups 
of p.c.f. fractals. For many fractals, lower bounds on the heat kernel
are also known, but we will not need these in this paper.  We also note that the heat kernel
is holomorphic on $\{\operatorname{Re}z>0\}$ (see
\cite{Kig_CUP01}). 

We mention that estimates of the form \eqref{eq:heat_estimates} are
also known for sums of even powers of vector fields on homogeneous
groups \cite{Heb_StudMath89}. However, we will not pursue these spaces
in this paper.

In this paper we write $A(y)\lesssim B(y)$ if there is a constant
$C$ \emph{independent} of $y$, but which might depend on the space
$X$, such that $A(y)\le CB(y)$ for all $y$. We write $A(y)\sim B(y)$
if $A(y)\lesssim B(y)$ and $B(y)\lesssim A(y)$. If $f(x,y)$ is
a function on $X_{1}\times X_{2}$, then we write $\Delta_{1}f$ to
denote the Laplacian of $f$ with respect to the first variable and
$\Delta_{2}f$ to denote the Laplacian of $f$ with respect to the
second variable; repeated subscripts indicate composition, for example
$\Delta_{21}=\Delta_{2}\circ\Delta_{1}$. We say that a function $u$
is \emph{smooth} if $u\in\operatorname{Dom}\Delta^{n}$ for all $n\ge0$.

\section{\label{sec:Symbols-and-pseudo-differential}Symbols and pseudo-differential
operators on fractals}

\global\long\def\D{\mathcal{D}}
\global\long\def\supp{\operatorname{supp}}
In this section $(X,R)$ is a metric measure space with a  Laplacian
$\Delta$ that satisfies the conditions \eqref{eq:double} and \eqref{eq:heat_estimates}.
We write $P(\lambda)$ for the spectral resolution of the positive
self-adjoint operator $-\Delta$. Examples of spaces for which the
results of this section apply include compact or infinite blow-ups
without boundary or products of copies of the same fractafold. In
the case that $X$ is either a compact fractafold without boundary
or an infinite fractafold without boundary for which the Laplacian
has pure point spectrum (\cite{Tep_JFA98,Sab:JFA00}), we write $P_{\lambda}$
for the spectral projection corresponding to the eigenvalue $\lambda$.
If $X$ is compact or $\Delta$ has pure point spectrum we fix
an orthonormal basis $\{\,\phi_{n}\,\}_{n\in\mathbb{N}}$ or $\{\,\phi_{n}\,\}_{n\in\mathbb{Z}}$
of $L^{2}(\mu)$ consisting of eigenfunctions with compact support
and write $\D$ for the dense set of finite linear combinations
with respect to this orthonormal basis. If $X$ is a product of such
fractals then there is a natural basis of $L^{2}(\mu)$ obtained
by taking products of eigenfunctions on each fiber (\cite{Str:TAM04}).
The following are the main objects of study in this paper.
\begin{defn}
For fixed $m\in\mathbb{R}$ we define the symbol class $S^{m}$ to
be the set of $p\in C^{\infty}((0,\infty))$ with the property that
for any integer $k\ge0$ there is $C_{k}>0$ such that
\[
\left\vert \bigl(\lambda\frac{d}{d\lambda}\bigr)^{k}p(\lambda)\right\vert \le C_{k}(1+\lambda)^{\frac{m}{d+1}}
\]
 for all $\lambda>0$, where $d$ is an in \eqref{eq:heat_estimates}.\end{defn}
\begin{rem}
The rationale for dividing $m$ by $d+1$ is that the Laplacian behaves
like an operator of order $d+1$.
\end{rem}
If $p$ is any bounded Borel function on $(0,\infty)$ then one can
define an operator $p(-\Delta)$ via
\[
p(-\Delta)u=\int_{0}^{\infty}p(\lambda)dP(\lambda)(u).
\]
This operator extends to a bounded operator on $L^{2}(\mu)$ by the
spectral theorem. If $p\in S^{m}$ with $m>0$, then $q(\lambda):=(1+\lambda)^{-\frac{m}{d+1}}p(\lambda)$
is bounded and one can define $p(-\Delta)=(I-\Delta)^{\frac{m}{d+1}}q(-\Delta)$.
\begin{defn}
\label{def:psdo_cc} For fixed $m\in\mathbb{R}$ define the class
$\Psi DO_{m}$ of pseudo-differential operators on $X$ to be the
collection of operators $p(-\Delta)$ with $p\in S^{m}$.
\end{defn}
If $X$ is a compact fractafold without boundary or $\Delta$ has
pure point spectrum then the formula for a pseudo-differential operator is
\[
p(-\Delta)u=\sum_{\lambda\in\Lambda}p(\lambda)P_{\lambda}u
\]
for $p\in S^{m}$ and $u\in\mathcal{D}$, where $\Lambda$ is the
spectrum of $-\Delta$.
\begin{prop}
[Symbolic Calculus]\label{lem:Symboliccalculus} If $p_{1}\in S^{m_{1}}$
and $p_{2}\in S^{m_{2}}$ then $p_{1}p_{2}\in S^{m_{1}+m_{2}}$ and
\[
p_{1}(-\Delta)\circ p_{2}(-\Delta)=p_{1}p_{2}(-\Delta).
\]
\end{prop}
\begin{proof}
Let $k\ge1$. Using Leibnitz's formula we have that
\[
\lambda^{k}\frac{d^{k}}{d\lambda^{k}}(p_{1}p_{2})(\lambda)=\sum_{j=0}^{k}\binom{k}{j}\lambda^{j}\frac{d^{j}}{d\lambda^{j}}p_{1}(\lambda)\lambda^{k-j}\frac{d^{k-j}}{d\lambda^{k-j}}p_{2}(\lambda).
\]
Therefore
\begin{align*}
\left\vert \lambda^{k}\frac{d^{k}}{d\lambda^{k}}(p_{1}p_{2})(\lambda)\right\vert  & \le\sum_{j=0}^{k}\binom{k}{j}(1+\lambda)^{\frac{m_{1}}{d+1}}(1+\lambda)^{\frac{m_{2}}{d+1}}\\
 & =2^{k}(1+\lambda)^{\frac{m_{1}+m_{2}}{d+1}}.
\end{align*}
Thus $p_{1}p_{2}\in S^{m_{1}+m_{2}}$. Using the identity $P(\lambda)(p(-\Delta)u)=p(\lambda)P(\lambda)(u)$
we have that
\begin{align*}
p_{1}(-\Delta)\circ p_{2}(-\Delta)u & =\int p_{1}(\lambda)P(\lambda)(p_{2}(-\Delta)u)\\
 & =\int p_{1}(\lambda)p_{2}(\lambda)P(\lambda)(u)=(p_{1}p_{2})(-\Delta)u.
\end{align*}

\end{proof}
The main result of this section is Theorem \ref{thm:smooth_kernel}, which says that if $p\in S^{0}$ then $p(-\Delta)$ is given by integration
with respect to a kernel that is smooth off the diagonal and satisfies
specific decay estimates off the diagonal. When $X$ is a
fractafold based on a p.c.f. fractal we then obtain from \cite[Theorem 1]{IR_2010} 
that $p(-\Delta)$ is a Calder\'on-Zygmund operator in the sense
of \cite[Section I.6.5]{Ste93}. We begin with a technical lemma.
\begin{lem}
\label{lem:ineq}Let  $\alpha>0$ and $R>0$ be fixed. Then
\begin{equation}
\sum_{n\in\mathbb{Z}}2^{\frac{n\alpha}{d+1}}\int\exp\bigl(-cR{}^{(d+1)\gamma}2^{n\gamma}(1+\xi^{2})^{-\frac{\gamma+1}{2}}\bigr)\frac{1}{(1+\xi^{2})^{j}}d\xi\le CR{}^{-\alpha}\label{eq:ineq}
\end{equation}
provided $j\ge\frac{\alpha(\gamma+1)}{2\gamma(d+1)}+\frac{1}{2}$.\end{lem}
\begin{proof}
Consider the integral
\begin{align*}
    I&= \int_{\mathbb{R}}
    \exp\bigl(-A(1+\xi^{2})^{-\frac{\gamma+1}{2}}\bigr)\frac{1}{(1+\xi^{2})^{j}}d\xi
     \\ 
    &=2\int_{0}^{\infty}
    \exp\bigl(-A(1+\xi^{2})^{-\frac{\gamma+1}{2}}\bigr)\frac{1}{(1+\xi^{2})^{j}}d\xi. 
    \end{align*}
Since $2(1+\xi^{2})\geq(1+\xi)^{2}$ on $(0,\infty)$ we have
\begin{equation*}
    I 
    \lesssim \int_{1}^{\infty} \exp(-\frac{A}{2}\xi^{-\gamma+1})\frac{d\xi}{\xi^{2j}}
    \simeq A^{\frac{1-2j}{\gamma+1}} \int_{0}^{A/2} e^{-t} t^{\frac{2j-1}{\gamma+1}} \frac{dt}{t}
    \leq C(j,\gamma) A^{\frac{1-2j}{\gamma+1}}
    \end{equation*}
for $j>\frac{1}{2}$.  We use this bound for $A\geq 1$ and the obvious bound by $C(j)$ for $A\leq 1$.  Let $A=cR^{(d+1)\gamma}2^{n\gamma}$ and $n_{0}$ be such that $A\geq 1$ iff $n\geq n_{0}$.  The series in the statement of the lemma is bounded by
\begin{equation*}
    C(j)\sum_{n< n_{0}} 2^{\frac{n\alpha}{d+1}} + C(j,\gamma) R^{\frac{(d+1)(1-2j)\gamma}{\gamma+1}}\sum_{n\geq n_{0}} 2^{n\bigl( \frac{\alpha}{d+1}-\frac{(2j-1)\gamma}{\gamma+1}\bigr)}
    \end{equation*}
which converges for $j$ as in the lemma.  The estimate follows from the fact that $2^{n_{0}}\simeq \frac{1}{c}R^{-(d+1)}$.
\end{proof}

\begin{thm}
\label{thm:smooth_kernel}Let $p:(0,\infty)\to\mathbb{C}$ be an $S^{0}$-symbol,
that is $p$ is smooth and for all $k\ge0$ there is $C_{k}>0$ such
that
\begin{equation}
\left\vert \lambda^{k}\frac{\partial^{k}}{\partial\lambda^{k}}p(\lambda)\right\vert \le C_{k}.\label{eq:symbol}
\end{equation}
Then $p(-\Delta)$ has a kernel $K(x,y)$ that is smooth off the diagonal
of $X\times X$ and satisfies
\begin{equation}
\vert K(x,y)\vert\lesssim R(x,y)^{-d}\label{eq:est1}
\end{equation}
and
\begin{equation}
\vert\Delta_{x}^{l}\Delta_{y}^{k}K(x,y)\vert\lesssim R(x,y)^{-d-(l+k)(d+1)}.\label{eq:est2}
\end{equation}
\end{thm}
\begin{proof}
We begin with the Littlewood-Paley dyadic decomposition from \cite[page 242]{Ste_PMS43_93}.
Let $\eta$ be a $C^{\infty}$ function with $\eta(\lambda)=1$
if $\vert\lambda\vert\le1$ and $\eta(\lambda)=0$ if $\vert\lambda\vert\ge2$
and let $\delta(\lambda)=\eta(\lambda)-\eta(2\lambda)$. Then \global\long\def\supp{\operatorname{supp}}
 $\supp\delta\subseteq\{\frac{1}{2}\le\vert\lambda\vert\le2\}$ and
\[
\sum_{n\in\mathbb{Z}}\delta(2^{-n}\lambda)=1,
\]
where for each $\lambda$ there are only two nonzero terms in the
above sum. Then $\delta$ is $C_{c}^{\infty}$ and we let $D_{k}>0$ for each $k\geq0$
such that
\begin{equation}
\left\vert \frac{d^{k}}{d\lambda^{k}}\delta(\lambda)\right\vert \le D_{k}.\label{eq:delta_bounds}
\end{equation}

For $n\in\mathbb{Z}$ let $p_{n}(\lambda)=p(\lambda)\delta(2^{-n}\lambda)$.
Then $\supp p_{n}\subseteq[2^{n-1},2^{n+1}]$ (since we assume that
$p$ is defined on $(0,\infty)$) and
\begin{equation}
p(\lambda)=\sum_{n\in\mathbb{Z}}p_{n}(\lambda).\label{eq:psumpn}
\end{equation}
Moreover we can use \eqref{eq:symbol} and the fact that the support
of $p_{n}$ has support in $[2^{n-1},2^{n+1}]$ to bound $\frac{d^{k}}{d\lambda^{k}}p_{n}$ by a constant
that depends only on $k$:
\begin{align}
    \left| \frac{d^{k}}{d\lambda^{k}}p_{n}(\lambda) \right|
    &=\left| \sum_{j=1}^{k}\binom{k}{j}\frac{d^{j}}{d\lambda^{j}}p(\lambda)\cdot2^{-n(k-j)}\frac{d^{k-j}}{d\lambda^{k-j}}\delta(2^{-n}\lambda)\right|\notag\\
    &\leq 2^{-nk}\sum_{j=0}^{k}\binom{k}{j}\frac{C_{j}}{\lambda^{j}}D_{k-j}2^{nj}\notag \\
    &\leq 2^{-nk}C(k)\sum_{j=1}^{k}\frac{2^{j}}{2^{nj}}2^{nj}=\overline{C}_{k}2^{-nk},\label{eq:estimates_pn}
    \end{align}
where $C(k)=\max C_{j}D_{j-k}$ and $\overline{C}_{k}=2^{k+1}C(k)$.

Fix now $n\in\mathbb{Z}$ and set $f_{n}(\lambda)=p_{n}(2^{n}\lambda)e^{\lambda}$.
Then $\supp f_{n}\subseteq[\frac{1}{2},2]$ and for all $k\ge0$ we
have that
\[
\frac{d^{k}}{d\lambda^{k}}f_{n}(\lambda)=\sum_{j=0}^{k}\binom{k}{j}2^{nj}\frac{d^{j}}{d\lambda^{j}}p_{n}(2^{n}\lambda)e^{\lambda}.
\]
Using \eqref{eq:estimates_pn} and the fact that $\lambda\le2$ we
obtain
\begin{equation}
\left\vert \frac{d^{k}}{d\lambda^{k}}f_{n}(\lambda)\right\vert \le\sum_{j=0}^{k}\binom{k}{j}2^{nj}\overline{C}_{j}2^{-nj}e^{2}=:A_{k}.\label{eq:estimates_fn}
\end{equation}

It follows immediately that there are constants $B_{k}$ independent of $n$ such that the Fourier transform $\hat{f}_{n}(\xi)$ of $f_{n}$ satisfies $|\xi^{k}\hat{f}_{n}(\xi)|\le B_{k}$, and thus for each $k\in\mathbb{Z}$
\begin{equation}
\vert\hat{f}_{n}(\xi)\vert\le\frac{\overline{D}_{k}}{(1+\xi^{2})^{k}}.\label{eq:fourier_bounds}
\end{equation}
for some constants $\bar{D}_{k}$ independent of $n$.

Now by Fourier inversion
\begin{equation*}
    p_{n}(\lambda)
    = e^{-\lambda2^{-n}} f_{n}(\lambda2^{-n})
    = e^{-\lambda2^{-n}} \frac{1}{2\pi} \int \hat{f}_{n}(\xi)e^{i \lambda\xi2^{-n}} d\xi
    \end{equation*}
so that
\begin{equation}
p_{n}(-\Delta)u=\frac{1}{2\pi}\int\hat{f}_{n}(\xi)e^{\Delta(1-i\xi)\frac{1}{2^{n}}}ud\xi\label{eq:pn_Delta}
\end{equation}
and the kernel of $p_{n}(-\Delta)$ is given by
\begin{equation}
K_{n}(x,y)=\frac{1}{2\pi}\int\hat{f}_{n}(\xi)h_{\frac{1}{2^{n}}-i\frac{\xi}{2^{n}}}(x,y)d\xi,\label{eq:kern_pn}
\end{equation}
where $h_{z}(x,y)$ is the complex heat kernel, which is holomorphic
on $\{\operatorname{Re}z>0\}$. A proof identical with that of Lemma
3.4.6 of \cite{Dav_CTM90} shows that there is $C>0$ such that
\begin{equation}
\vert h_{z}(x,y)\vert\le C(\operatorname{Re}z)^{-\frac{d}{d+1}}.\label{eq:complex_ker_est1}
\end{equation}
Using the above estimate together with the heat kernel estimates \eqref{eq:heat_estimates},
Lemma 9 of \cite{Dav_JFA95} implies that
\begin{equation}
\vert h_{z}(x,y)\vert\le C2^{\frac{d}{d+1}}(\vert z\vert\cos(\theta))^{-\frac{d}{d+1}}e^{-\frac{c\gamma R(x,y)^{(d+1)\gamma}}{2}\vert z\vert^{-\gamma}\cos(\theta)},\label{eq:complex_heat_kernel_est}
\end{equation}
where $z=r(\cos\theta+i\sin\theta)$. For $z=\frac{1}{2^{n}}-i\frac{\xi}{2^{n}}$
we have that $\vert z\vert=2^{-n}(1+\xi^{2})^{\frac{1}{2}}$, $\cos\theta=(1+\xi^{2})^{-\frac{1}{2}}$,
and $\vert z\vert\cos\theta=2^{-n}$. Thus
\begin{equation}
\left\vert h_{\frac{1}{2^{n}}-i\frac{\xi}{2^{n}}}(x,y)\right\vert \le C2^{\frac{d}{d+1}}2^{\frac{nd}{d+1}}\exp\left(-\frac{c\gamma}{2}R(x,y)^{(d+1)\gamma}2^{n\gamma}(1+\xi^{2})^{-\frac{\gamma+1}{2}}\right),\label{eq:heat_ker_est2}
\end{equation}
and
\[
\vert K_{n}(x,y)\vert\le C(d,j)2^{\frac{nd}{d+1}}\int\exp\left(-\frac{c\gamma}{2}R(x,y)^{(d+1)\gamma}2^{n\gamma}(1+\xi^{2})^{-\frac{\gamma+1}{2}}\right)\frac{1}{(1+\xi^{2})^{j}}d\xi,
\]
for all $j\ge1$, where the constants $C(d,j)$ depend only on $d$
and $j$. Lemma \ref{lem:ineq} with $\alpha=d$ and $R=R(x,y)$ implies
that if we choose $j$ large enough then $K(x,y)=\sum_{n\in\mathbb{Z}}K_{n}(x,y)$
is defined and continuous off the diagonal and it satisfies \eqref{eq:est1}.
Finally, using \eqref{eq:psumpn}, we see that $K(x,y)$ is the kernel
of $p(-\Delta)$.

Next we want to prove the $K(x,y)$ is smooth off the diagonal and
show that the estimates \eqref{eq:est2} hold. For simplicity we give the complete argument only in the case $l=0$ and $k=1$, i.e. we show that
\begin{equation}
\vert\Delta_{y}K(x,y)\vert\lesssim R(x,y)^{-2d-1}.\label{eq:eq2_particular}
\end{equation}
Let $q(\lambda)=\lambda p(\lambda)$ and $q_{n}(\lambda)=\lambda p_{n}(\lambda)$
for all $n\in\mathbb{Z}$. By \eqref{eq:estimates_pn}
we obtain that $|q_{n}^{(k)}(\lambda)|\le C(k)2^{(1-k)n}$
for all $k\ge0$. If we set $g_{n}(\lambda)=2^{-n}q_{n}(2^{n}\lambda)e^{\lambda}$ one then obtains $|g_{n}^{(k)}(\lambda)|\le A(k)$
where $A(k)$ are constants independent of $n$. Then a computation similar to \eqref{eq:fourier_bounds} shows that
\begin{equation}
\vert\hat{g}_{n}(\xi)\vert\le\frac{B(k)}{(1+\xi^{2})^{k}},\label{eq:est_gn}
\end{equation}
for all $k\ge0$, with $B(k)$ independent of $n$. Using the Fourier
inversion formula we have that
\[
q_{n}(\lambda)=\frac{1}{2\pi}2^{n}\int\hat{g}_{n}(\xi)e^{-\lambda\left(\frac{1}{2^{n}}-i\frac{\xi}{2^{n}}\right)}d\xi,
\]
and a similar formula holds for $q_{n}(-\Delta)$.  Hence
\[
\Delta_{y}K_{n}(x,y)=\frac{1}{2\pi}2^{n}\int\hat{g}_{n}(\xi)h_{\frac{1}{2^{n}}-i\frac{\xi}{2^{n}}}(x,y)d\xi,
\]
and so
\begin{align*}
\vert\Delta_{y}K_{n}(x,y)\vert & \le C(d)2^{n}2^{\frac{nd}{d+1}}\int\vert\hat{g}_{n}(\xi)\vert\exp\left(-c(\gamma)\frac{R(x,y)^{(d+1)\gamma}2^{n\gamma}}{(1+\xi^{2})^{\frac{\gamma+1}{2}}}\right)d\xi\\
 & \le  C(d,j)2^{\frac{n(2d+1)}{d+1}}\int\exp\left(-c(\gamma)\frac{R(x,y)^{(d+1)\gamma}2^{n\gamma}}{(1+\xi^{2})^{\frac{\gamma+1}{2}}}\right)\frac{1}{(1+\xi^{2})^{j}}d\xi,
\end{align*}
for all $j\ge1$. Using Lemma \ref{lem:ineq} with $\alpha=2d+1$
we find for $j$ large enough that $\Delta_{y}K(x,y)=\sum_{n}\Delta_{y}K_{n}(x,y)$
is well-defined and continuous off the diagonal and it satisfies \eqref{eq:eq2_particular}.

To prove \eqref{eq:est2} for general $l$ and $k$, one
can repeat the above steps for the functions $q(\lambda)=\lambda^{l+k}p(\lambda)$,
$q_{n}(\lambda)=\lambda^{l+k}p_{n}(\lambda)$, $g_{n}(\lambda)=2^{-n(l+k)}q_{n}(\lambda)e^{\lambda}$
and apply Lemma \ref{lem:ineq} with $\alpha=d+(l+k)(d+1)$.
\end{proof}
\begin{cor}
\label{pro:Prop_boundedness}Assume that $X$ is a fractafold without
boundary or a product of such fractafolds. If $p\in S^{0}$ then $p(-\Delta)$
is a Calder\'on-Zygmund operator and, thus, it extends to a bounded
operator on $L^{q}(\mu)$ for all $1<q<\infty$ and satisfies weak
$1-1$ estimates.\end{cor}
\begin{proof}
Using the estimates \eqref{eq:est1} and \eqref{eq:est2} (with $l=0$
and $k=1$), Theorem 1.1 of \cite{IR_2010} implies that $p(-\Delta)$
is a Calder\'on-Zygmund operator.\end{proof}
\begin{rem}
The boundedness of $p(-\Delta)$ on $L^{q}(\mu)$, $1<q<\infty$,
can be obtained also using the results of \cite{DuOuSi_JFA02}.\end{rem}
\begin{cor}
\label{cor:smooth_kernel_m}If $p\in S^{m}$ then $p(-\Delta)$ is
given by integration with respect to a kernel $K_{p}$ that is smooth
off the diagonal.\end{cor}
\begin{proof}
If $p\in S^{m}$ then $q(\lambda)=p(\lambda)(1+\lambda)^{-m}\in S^{0}$
and Theorem \ref{thm:smooth_kernel} implies that $q(-\Delta)$ has
a kernel $K_{q}$ that is smooth off the diagonal. Then
\[
K_{p}(x,y)=(I-\Delta_{x})^{m}K_{q}(x,y)
\]
is smooth off the diagonal and it is the kernel of $p(-\Delta)$.
\end{proof}
As a consequence, we obtain that, for $\operatorname{Re}s\ge0$, the
Bessel potentials
\begin{equation}
(I-\Delta)^{-s}u=\sum_{\lambda\in\Lambda}(1+\lambda)^{-s}P_{\lambda}u,\label{eq:Bessel}
\end{equation}
and the Riesz potentials
\[
(-\Delta)^{-s}u=\sum_{\lambda\in\Lambda}\lambda{}^{-s}P_{\lambda}u,
\]
have smooth kernels and are bounded on $L^{p}(\mu)$, for $1<p<\infty$,
with operator norm of at most polynomial growth in $\operatorname{Im}s$
(when $\operatorname{Re}s=0$). These facts were previously proved in \cite{IR_2010}.

\section{\label{sec:Sobolev-Spaces}Sobolev Spaces}

In this section we assume that $X$ is a compact fractafold without
boundary based on a p.c.f. fractal $K$, an infinite blow-up of $K$
without boundary, or a product of copies of the same fractafold without
boundary. Moreover, we assume that if $X$ is non-compact then $\Delta$
has pure point spectrum. Therefore we can consider that $L^{2}(\mu)$
is spanned by compactly supported eigenfunctions of $-\Delta$. As
before, we fix an orthonormal basis $\{\phi_{n}\}$ of $L^{2}(\mu)$
consisting of such eigenfunctions and we write $\D$ for the space
of finite linear combinations of $\phi_{n}$'s.
\begin{defn}
[$L^2$-Sobolev spaces]\label{def:L2Sobolevspaces} For $s\ge0$ and
$u\in L^{2}$ define
\[
\Vert u\Vert_{H^{s}}^{2}:=\sum_{\lambda\in\Lambda}(1+\lambda)^{\frac{2s}{d+1}}\Vert P_{\lambda}u\Vert_{2}^{2}.
\]
We say that $u\in H^{s}$ if and only if $\Vert u\Vert_{H^{s}}<\infty$.

For $s\ge0$ if $u\in H^{s}$ then $u\in L^{2}(\mu)$. So for $s\ge0$
there is no harm in starting with $u\in L^{2}(\mu)$ in Definition
\ref{def:L2Sobolevspaces}. Moreover if $s_{1}\le s_{2}$ then $H^{s_{2}}\subseteq H^{s_{1}}$.

For $s<0$ we define $H^{s}$ using the distribution theory developed
in \cite{RogStr-2009_distr}. Namely, for $s>0$, $H^{-s}$ is the
dual of $H^{s}$ via
\[
\langle f,\varphi\rangle=\sum\langle f,\phi_{n}\rangle\overline{\langle\varphi,\phi_{n}\rangle}
\]
for $f\in H^{-s}$ and $\varphi\in H^{s}$. The fact that the above
linear functional is bounded is a consequence of the Cauchy-Schwartz
inequality:
\[
\langle f,\varphi\rangle=\sum_{n}\lambda_{n}^{\frac{-s}{d+1}}\langle f,\phi_{n}\rangle\lambda_{n}^{\frac{s}{d+1}}\overline{\langle\varphi,\phi_{n}\rangle}\le\Vert f\Vert_{H^{-s}}\Vert\varphi\Vert_{H^{s}}.
\]
\end{defn}
\begin{lem}
For $s$ a positive integer, $u\in H^{s(d+1)}$ if and only if $(I-\Delta)^{k}u\in L^{2}(\mu)$
for all $k\le s$.\end{lem}
\begin{proof}
Notice that
\[
\Vert u\Vert_{H^{s(d+1)}}^{2}=\sum_{\lambda\in\Lambda}(1+\lambda)^{2s}\Vert P_{\lambda}u\Vert_{2}^{2}=\Vert(I-\Delta)^{s}u\Vert_{2}^{2}.
\]
From this equality the statement follows.\end{proof}
\begin{prop}
If $p\in S^{m}$ then $p(-\Delta):H^{s}\to H^{s-m}$ for all $s$
and $m$.\end{prop}
\begin{proof}
Recall that $P_{\lambda}(p(-\Delta_{})u)=p(\lambda)P_{\lambda}u$.
Then
\begin{align*}
\Vert p(-\Delta)u\Vert_{H^{s-m}}^{2} & =\sum_{\lambda\in\Lambda}(1+\lambda)^{\frac{2(s-m)}{d+1}}\Vert P_{\lambda}(p(-\Delta)u)\Vert_{2}^{2}\\
 & =\sum_{\lambda\in\Lambda}(1+\lambda)^{\frac{2(s-m)}{d+1}}\vert p(\lambda)\vert^{2}\Vert P_{\lambda}u\Vert_{2}^{2}\\
 & \lesssim\sum_{\lambda\in\Lambda}(1+\lambda)^{\frac{2s}{d+1}}\Vert P_{\lambda}u\Vert_{2}^{2}=\Vert u\Vert_{H^{s}}^{2}.
\end{align*}
Thus $p(-\Delta)u\in H^{s-m}$ whenever $u\in H^{s}$.\end{proof}
\begin{prop}
\label{lem:domain_Hs}For $s\ge0$, $H^{s}$ equals the image of $L^{2}(\mu)$
under $(I-\Delta)^{-s/(d+1)}$.\end{prop}
\begin{proof}
Let $u,f\in L^{2}(\mu)$ such that $u=(I-\Delta)^{-s/(d+1)}f$. Then
\[
P_{\lambda}u=\frac{1}{(1+\lambda)^{s/(d+1)}}P_{\lambda}f.
\]
Therefore
\[
\Vert u\Vert_{H^{s}}^{2}=\sum_{\lambda\in\Lambda}(1+\lambda)^{\frac{2s}{d+1}}\frac{1}{(1+\lambda)^{2s/(d+1)}}\Vert P_{f}f\Vert_{2}^{2}=\Vert f\Vert_{2}^{2}<\infty.
\]
Now let $u\in H^{s}$ and define $f=(I-\Delta)^{s/(d+1)}u$, that
is
\[
f=\sum(1+\lambda)^{\frac{s}{d+1}}P_{\lambda}u.
\]
Then
\begin{align*}
\Vert f\Vert_{2}^{2} & =\sum_{\lambda\in\Lambda}(1+\lambda)^{\frac{2s}{d+1}}\Vert P_{\lambda}u\Vert_{2}^{2}=\Vert u\Vert_{H^{s}}^{2}<\infty.
\end{align*}
 So $f\in L^{2}(\mu)$ and $u=(I-\Delta)^{-s/(d+1)}f$.
\end{proof}
We extend next the definition of Sobolev spaces to $L^{p}$ spaces,
using Proposition \ref{lem:domain_Hs} as the starting point. Namely,
we are going to replace $L^{2}(\mu)$ with $L^{p}(\mu)$ (see \cite{Ste_PMS30_70}).
\begin{defn}
We define the $L^{p}$ Sobolev spaces $L_{s}^{p}$ for $s\ge0$ and
$1<p<\infty$ to be the images of $L^{p}$ under $(I-\Delta)^{-s/(d+1)}$
with the norm
\[
\Vert(I-\Delta)^{-s/(d+1)}f\Vert_{L_{s}^{p}}=\Vert f\Vert_{L^{p}(\mu)}.
\]

\end{defn}
Note that by Corollary \ref{pro:Prop_boundedness} we may regard $L_{s}^{p}$
as a closed subspace of $L^{p}$, with $\Vert u\Vert_{p}\le c\Vert u\Vert_{L_{s}^{p}}$.
For $s=0$ we have $L_{0}^{p}=L^{p}$ and for $p=2$ we have $L_{s}^{2}=H^{s}$.
\begin{prop}
If $0<s_{0}\le s_{1}<\infty$ and $1<p_{0},p_{1}<\infty$ then the
complex interpolation space $\bigl[L_{s_{0}}^{p_{0}},L_{s_{1}}^{p_{1}}\bigr]_{\theta}$
may be identified with $L_{s}^{p}$ where $0<\theta<1$, $s=(1-\theta)s_{0}+\theta s_{1}$,
and $\frac{1}{p}=\frac{1-\theta}{p_{0}}+\frac{\theta}{p_{1}}$.\end{prop}
\begin{proof}
The proof follows as in the Euclidean case using the comments following
Proposition \ref{pro:Prop_boundedness}.\end{proof}
\begin{prop}
\label{pro:bound_L_p_s}If $p\in S^{m}$ then $p(-\Delta)$ is a bounded
operator from $L_{s}^{p}$ into $L_{s-m}^{p}$.\end{prop}
\begin{proof}
Let $f\in L_{s}^{p}$. Then there is $g\in L^{p}$ such that $f=(I-\Delta)^{-s/(d+1)}g$.
Then $p(-\Delta)f=p(-\Delta)(I-\Delta)^{-s/(d+1)}g$. By Proposition
\ref{lem:Symboliccalculus} and Corollary \ref{pro:Prop_boundedness}
\[
(I-\Delta)^{(s-m)/(d+1)}p(-\Delta)(I-\Delta)^{-s/(d+1)}g\in L^{p}.
\]
This is equivalent with $p(-\Delta)f=(I-\Delta)^{-(s-m)/(d+1)}g\in L_{s-m}^{p}$.
 \end{proof}
\begin{lem}
The Sobolev space $L_{d+1}^{p}$ equals $dom_{L^{p}}(\Delta)$.\end{lem}
\begin{proof}
Let $u\in dom_{L^{p}}(\Delta)$. Then $u\in dom(\mathcal{E})$ and
there is $f\in L^{p}$ such that $f=-\Delta u$. That is
\[
\int fvd\mu=\mathcal{E}(u,v)\;\mbox{for all }v\in dom(\mathcal{E}).
\]
If $v=\phi_{\lambda}$, where $\phi_{\lambda}$ is any eigenfunction
corresponding to $\lambda$, then $\lambda\int u\phi_{\lambda}d\mu=\int f\phi_{\lambda}d\mu$.
This implies that $(I-\Delta)^{-1}(f+u)=u$.

For the converse, notice first that if $\lambda_{1},\lambda_{2}\in\Lambda$
and $\phi_{\lambda_{i}}$ is any eigenfunction corresponding to $\lambda_{i}$,
$i=1,2$, then $\mathcal{E}(\phi_{\lambda_{1}},\phi_{\lambda_{2}})=\lambda_{1}\delta(\lambda_{1},\lambda_{2})$.
Let $u=(I-\Delta)^{-1}g$ for some $g\in L^{p}$. Set $f=g-u\in L^{p}$.
For any $\lambda\in\Lambda$ we have
\[
\int f\phi_{\lambda}d\mu=\int(g-u)\phi_{\lambda}d\mu=\frac{\lambda}{\lambda+1}\int g\phi_{\lambda}d\mu.
\]
Also
\[
\mathcal{E}(u,\phi_{\lambda})=\frac{1}{1+\lambda}\int g\phi_{\lambda}d\mu\,\mathcal{E}(\phi_{\lambda},\phi_{\lambda})=\frac{\lambda}{1+\lambda}\int g\phi_{\lambda}d\mu.
\]
It follows by linearity and density that $\int fvd\mu=\mathcal{E}(u,v)$
for all $v\in dom\mathcal{E}$. Thus $-\Delta u=f$.\end{proof}
\begin{thm}
[Sobolev embedding theorem]If $s<\frac{d}{p}$ then $L_{s}^{p}\subseteq L^{q}$
for $\frac{1}{q}=\frac{1}{p}-\frac{s}{d}$.\end{thm}
\begin{proof}
We have all the ingredients needed to use the same proof as in \cite[Theorem 3.11]{Str_JFA03}.
\end{proof}

\section{\label{sec:products}Pseudo-differential operators on Products of
Fractals}

\global\long\def\proj#1{P_{\lambda_{1}}\otimes P_{\lambda_{2}}\dots\otimes P_{\lambda_{N}}#1}
In this section we extend the definition of pseudo-differential operators
to products of metric measure spaces that satisfy the hypothesis of
Section \ref{sec:background}. Let $N\ge2$ be fixed and let $(X_{1},R_{1}),\dots,(X_{N},R_{N})$
be $N$ metric spaces such that $X_{i}$ has measure $\mu_{i}$ and
 Laplacian $\Delta_{i}$. In case that $X_{i}$ is a fractafold, then
we consider that $\Delta_{i}$ is defined using a self-similar Dirichlet
energy $\E_{i}$. We assume that the heat kernel $h^{(i)}$ associated
to $\Delta_{i}$ satisfies the estimates \eqref{eq:heat_estimates}
with $d=d_{i}$, $\gamma=\gamma_{i}$, $R=R_{i}$, for all $i=1,\dots,N$.
Let $X=X_{1}\times\dots\times X_{N}$ be the product space and let
$\mu=\mu_{1}\times\dots\times\mu_{N}$ be the product measure on $X$.
We write $x=(x_{1},\dots,x_{n})$ for elements in $X$, $\lambda=(\lambda_{1},\dots,\lambda_{N})$
for elements in $(0,\infty)^{N}$ and $\Delta=(\Delta_{1},\dots,\Delta_{N})$.
Recall that there is a unique spectral decomposition $P_{\lambda}$
such that
\begin{equation}
P_{\lambda}u_{1}\times u_{2}\times\dots\times u_{N}=\proj{u_{1}\times\dots\times u_{N}},\label{eq:spectral_proj_mv}
\end{equation}
where $u_{1}\times\dots\times u_{n}(x)=u_{1}(x_{1})u_{2}(x_{2})\dots u_{N}(x_{N})$
and
\[
\proj{u_{1}(x_{1})u_{2}(x_{2})\dots u_{N}(x_{N})}=P_{\lambda_{1}}u_{1}(x_{1})P_{\lambda_{2}}u_{2}(x_{2})\dots P_{\lambda_{N}}u_{N}(x_{N}),
\]
$P_{\lambda_{i}}$ is the spectral projection corresponding to the
eigenvalue $\lambda_{i}$ of $-\Delta_{i}$.

Recall from \cite{Str:TAM04} that if $\{\phi_{n}^{i}\}$ is an orthonormal
basis for $L^{2}(\mu_{i})$ consisting of compactly supported eigenfunctions of $-\Delta_{i}$,
then
\[
\phi_{k_{1},\dots,k_{N},}(x)=\phi_{k_{1}}^{1}(x_{1})\cdot\dots\cdot\phi_{k_{N}}^{N}(x_{N}),
\]
$k_{i}\in\mathbb{Z}$ for all $i=1,\dots,N$, form an orthonormal
basis for $L^{2}(\mu)$. In this case we write $\mathcal{D}^{N}$
be the set of finite linear combinations of $\{\phi_{k_{1},\dots,k_{N}}\}$.
\begin{defn}
\label{def:psdo_cc_prod}For $m\in\mathbb{Z}$ we define the symbol
class $S^{m}$ on $X$ to consist of the set of smooth functions $p:(0,\infty)^{N}\to\mathbb{C}$
that satisfy the property that for every $\alpha=(\alpha_{1},\dotsm,\alpha_{N})\in\mathbb{N}^{N}$
there is a positive constant $C_{\alpha}$ such that
\begin{equation}
\lambda{}^{\alpha}\left\vert \frac{\partial^{\alpha}}{\partial\lambda^{\alpha}}p(\lambda)\right\vert \le C_{\alpha}(1+\lambda)^{\frac{m}{d+1}},\label{eq:symbol_mul_m}
\end{equation}
where $\lambda^{\alpha}=\lambda_{1}^{\alpha_{1}}\lambda_{2}^{\alpha_{2}}\cdots\lambda_{N}^{\alpha_{N}}$,
and
\[
\frac{\partial^{\alpha}}{\partial\lambda^{\alpha}}p(\lambda)=\frac{\partial^{\alpha_{1}}}{\partial\lambda_{1}^{\alpha_{1}}}\dots\frac{\partial^{\alpha_{N}}}{\partial\lambda_{N}^{\alpha_{N}}}p(\lambda_{1},\dots,\lambda_{N}).
\]
\end{defn}
\begin{rem}
Note that each $\lambda_{i}$ is positive and we omit the absolute
values of $\lambda$ in the above definition. Notice also that if
$p$ satisfies the condition \eqref{eq:symbol_mul_m} then it satisfies
also the condition
\begin{equation}
\vert\lambda\vert^{\vert\alpha\vert}\left\vert \frac{\partial^{\alpha}}{\partial\lambda^{\alpha}}p(\lambda)\right\vert \le C_{\alpha}(1+\lambda)^{\frac{m}{d+1}},\label{eq:symbol_mul_alt}
\end{equation}
where $\vert\lambda\vert=\sqrt{\lambda_{1}^{2}+\dots+\lambda_{N}^{2}}$
if $\lambda\in(0,\infty)^{N}$, and $\vert\alpha\vert=\alpha_{1}+\dots+\alpha_{N}$,
if $\alpha\in\mathbb{N}^{N}$. Condition \eqref{eq:symbol_mul_m}
is usually called the Marcinkiewicz condition while
\eqref{eq:symbol_mul_alt} is the H\"ormander condition.
Many authors define psuedo-differential operators on Euclidean spaces
or manifolds using the H\"ormander condition. We chose to use Marcinkiewicz
condition in the definition because it is more general and it makes
the proof of the main theorem in this section more transparent.
\end{rem}
If $p:(0,\infty)^{N}\to\mathbb{C}$ is a bounded Borel function then
we can define an operator $p(-\Delta)$ acting on $L^{2}(\mu)$ via
\begin{equation}
p(-\Delta)u=\int_{(0,\infty)^{N}}p(\lambda)P_{\lambda}(u),\label{eq:psdo_mv}
\end{equation}
where $P_{\lambda}$ is defined via \eqref{eq:spectral_proj_mv}.
\begin{defn}
For $m\in\mathbb{Z}$ define the class $\Psi DO_{m}^{N}$ of pseudo-differential
operators on $X$ to be the collection of operators $p(-\Delta)$
with $p\in S^{m}$.
\end{defn}
The spectral theorem implies that, if $m=0$, then $p(-\Delta)$ extends
to $L^{2}(\mu)$. Moreover we will show that, as in the single
variable case, these operators are given by integration with respect
to kernels that are smooth off the diagonal and, if $X$ is a product
of fractafolds, they are Calder\'on-Zygmund operators on $X$.  The following is the main theorem of this section.
\begin{thm}
\label{thm:smooth_ker_mul}Suppose that $p\in S^{0}$, that is, $p$
is smooth and
\begin{equation}
\lambda^{\alpha}\left\vert \frac{\partial^{\alpha}}{\partial\lambda^{\alpha}}p(\lambda)\right\vert \le C_{\alpha}\label{eq:symbol_mul}
\end{equation}
for all $\alpha=(\alpha_{1},\dotsm,\alpha_{N})\in\mathbb{N}^{N}$.
Then $p(-\Delta)$, where $\Delta=(\Delta_{1},\dots,\Delta_{N})$,
is given by integration with respect to a kernel $K_{p}$ that is
smooth off the diagonal and satisfies
\begin{equation}
\vert K_{p}(x,y)\vert\lesssim\prod_{k=1}^{N}R_{k}(x_{k},y_{k})^{-d_{k}}\label{eq:1_mul}
\end{equation}
and
\begin{equation}
\left\vert \Delta_{x,i}^{\beta_{1}}\Delta_{y,j}^{\beta_{2}}K_{p}(x,y)\right\vert \lesssim R_{i}(x_{i}y_{i})^{-\beta_{1}(d_{i}+1)}R_{j}(x_{j},y_{j})^{-\beta_{2}(d_{j}+1)}\prod_{k=1}^{N}R(x_{k},y_{k})^{-d_{k}},\label{eq:2_mul}
\end{equation}
for all $i,j=1,\dots,N$, where $x=(x_{1},\dots,x_{N}),y=(y_{1},\dots,y_{N})$
and $\Delta_{x,i}$ is the Laplacian with respect to $x_{i}$ and
$\Delta_{y,j}$ is the Laplacian with respect to $y_{j}$.\end{thm}
\begin{proof}
For simplicity we  prove the result for $N=2$. The difference
between this and the general case is a matter of notation. We use the Littlewood-Paley decomposition from \cite[page 242]{Ste_PMS43_93}
in each variable. Let $\delta$ be the smooth function from the proof of Theorem \ref{thm:smooth_kernel} and let $D_{k}$
be the bounds from \eqref{eq:delta_bounds}. For $n,m\in\mathbb{Z}$
define
\[
p_{n,m}(\lambda)=p(\lambda)\delta(2^{-n}\lambda_{1})\delta(2^{-m}\lambda_{2}),
\]
where $\lambda=(\lambda_{1},\lambda_{2})\in(0,\infty)^{2}$. Then
$p_{n,m}$ is a smooth function with
\begin{equation}
\supp p_{n,m}\subseteq[2^{n-1},2^{n+1}]\times[2^{m-1},2^{m+1}],\label{eq:supp_pnm}
\end{equation}
for all $n,m\in\mathbb{Z}$.

Let $n,m\in\mathbb{Z}$ be fixed. We use Leibnitz formula to estimate
the bounds on the derivatives of $p_{n,m}$. If $\alpha=(\alpha_{1},\alpha_{2})\in\mathbb{N}^{2}$
then we have
\begin{align*}
\frac{\partial^{\alpha}}{\partial\lambda^{\alpha}}p_{n,m}(\lambda) & =  \frac{\partial^{\alpha_{1}}}{\partial\lambda_{1}^{\alpha_{1}}}\frac{\partial^{\alpha_{2}}}{\partial\lambda_{2}^{\alpha_{2}}}\bigl(p(\lambda)\delta(2^{-n}\lambda_{1})\delta(2^{-m}\lambda_{2})\bigr)\\
 & =  \frac{\partial^{\alpha_{1}}\delta(2^{-n}\lambda_{1})}{\partial\lambda_{1}^{\alpha_{1}}}\left(\sum_{k=0}^{\alpha_{2}}\left(\begin{array}{c}
\alpha_{2}\\
k
\end{array}\right)\frac{\partial^{k}p(\lambda)}{\partial\lambda_{2}^{k}}2^{-m(\alpha_{2}-k)}\delta^{(\alpha_{2}-k)}(2^{-m}\lambda_{2})\right)\\
 & =  \sum_{j=0}^{\alpha_{1}}\sum_{k=0}^{\alpha_{2}}\left[\left(\begin{array}{c}
\alpha_{1}\\
j
\end{array}\right)\left(\begin{array}{c}
\alpha_{2}\\
k
\end{array}\right)\left(\frac{\partial^{(j,k)}}{\partial\lambda^{(j,k)}}p(\lambda)\right)2^{-n(\alpha_{1}-j)}2^{-m(\alpha_{2}-k)}\right.\\
 & \cdot  \left.\delta^{(\alpha_{1}-j)}(2^{-n}\lambda_{1})\delta^{(\alpha_{2}-k)}(2^{-m}\lambda_{2})\right].
\end{align*}
Therefore, using \eqref{eq:symbol_mul} we obtain that
\begin{align}
\left\vert \frac{\partial^{\alpha}}{\partial\lambda^{\alpha}}p_{n,m}(\lambda)\right\vert  & \le  2^{-n\alpha_{1}}2^{-m\alpha_{2}}\sum_{j,k}\left(\begin{array}{c}
\alpha_{1}\\
j
\end{array}\right)\left(\begin{array}{c}
\alpha_{2}\\
k
\end{array}\right)\frac{C_{(j,k)}2^{nj}2^{mk}}{\lambda_{1}^{j}\lambda_{2}{}^{k}}\nonumber \\
 & \le  C(\alpha)2^{-n\alpha_{1}}2^{-m\alpha_{2}}\sum_{j,k}\left(\begin{array}{c}
\alpha_{1}\\
j
\end{array}\right)\left(\begin{array}{c}
\alpha_{2}\\
k
\end{array}\right)\frac{2^{nj}2^{mk}2^{j+k}}{2^{nj}2^{mk}}\nonumber \\
 & \le  C(\alpha)2^{-n\alpha_{1}}2^{-m\alpha_{2}}\sum_{j,k}\left(\begin{array}{c}
\alpha_{1}\\
j
\end{array}\right)\left(\begin{array}{c}
\alpha_{2}\\
k
\end{array}\right)2^{j+k}\nonumber \\
 & =  3^{\alpha_{1}+\alpha_{2}}C(\alpha)2^{-n\alpha_{1}}2^{-m\alpha_{2}}=:\overline{C}_{\alpha}2^{-n\alpha_{1}}2^{-m\alpha_{2}},\label{eq:est_pnm}
\end{align}
where $C(\alpha)=\max C_{(j,k)}$ and we used in going from the first
line to the second that $\lambda_{1}\ge2^{n-1}$ and $\lambda_{2}\ge2^{m-1}$.
Define now $f_{n,m}(\lambda)=p_{n,m}(2^{n}\lambda_{1},2^{m}\lambda_{2})e^{\lambda_{1}+\lambda_{2}}$.
Then
\begin{equation}
\supp f_{n,m}(\lambda)\subseteq\left[\frac{1}{2},2\right]\times\left[\frac{1}{2},2\right].\label{eq:supp_fnm}
\end{equation}
If $\alpha=(\alpha_{1},\alpha_{2})\in\mathbb{N}^{2}$ then using Leibnitz's
rule one more time we obtain that
\begin{align*}
\frac{\partial^{\alpha}}{\partial\lambda^{\alpha}}f_{n,m}(\lambda) & =  \frac{\partial^{\alpha_{1}}}{\partial\lambda^{\alpha_{1}}}\frac{\partial^{\alpha_{2}}}{\partial\lambda_{2}^{\alpha_{2}}}\bigl(p_{n,m}(2^{n}\lambda_{1},2^{m}\lambda_{2})e^{\lambda_{1}+\lambda_{2}}\bigr)\\
 & =  \frac{\partial^{\alpha_{1}}}{\partial\lambda^{\alpha_{1}}}\left(\sum_{k=0}^{\alpha_{2}}\left(\begin{array}{c}
\alpha_{2}\\
k
\end{array}\right)2^{mk}\left(\frac{\partial^{k}}{\partial\lambda_{2}^{k}}p_{n,m}(2^{n}\lambda_{1},2^{m}\lambda_{2})\right)e^{\lambda_{1}+\lambda_{2}}\right)\\
 & =  \sum_{j=0}^{\alpha_{1}}\sum_{k=0}^{\alpha_{2}}\left(\begin{array}{c}
\alpha_{1}\\
j
\end{array}\right)\left(\begin{array}{c}
\alpha_{2}\\
k
\end{array}\right)2^{nj}2^{mk}\frac{\partial^{(j,k)}}{\partial\lambda^{(j,k)}}p_{n,m}(2^{n}\lambda_{1},2^{m}\lambda_{2})e^{\lambda_{1}+\lambda_{2}}.
\end{align*}
Thus, using \eqref{eq:est_pnm} and \eqref{eq:supp_fnm} we obtain
that
\begin{eqnarray}
\left\vert \frac{\partial^{\alpha}}{\partial\lambda^{\alpha}}f_{n,m}(\lambda)\right\vert  & \le & \sum_{j=0}^{\alpha_{1}}\sum_{k=0}^{\alpha_{2}}\left(\begin{array}{c}
\alpha_{1}\\
j
\end{array}\right)\left(\begin{array}{c}
\alpha_{2}\\
k
\end{array}\right)2^{nj}2^{mk}\overline{C}_{(j,k)}2^{-nj}2^{-mk}e^{4}\nonumber \\
 & \le & \overline{C}(\alpha)2^{\alpha_{1}+\alpha_{2}}e^{4}=:A_{\alpha}.\label{eq:est_fnm}
\end{eqnarray}
Therefore the Fourier transform of $f_{n,m}$ satisfies
\[
\left\vert \left(\frac{\partial^{\alpha}}{\partial\lambda^{\alpha}}f_{n,m}\right)^{\wedge}(\xi)\right\vert \le B_{\alpha}
\]
for all $\alpha\in\mathbb{N}^{2}$, where $B_{\alpha}>0$, and, thus,
\begin{equation}
\bigl\vert\widehat{f}_{n,m}(\xi)\bigr\vert\le\frac{\overline{D}_{\alpha}}{(1+\xi_{1}^{2})^{\alpha_{1}}(1+\xi_{2}^{2})^{\alpha_{2}}}\label{eq:est_Ffnm}
\end{equation}
for all $\xi=(\xi_{1},\xi_{2})\in\mathbb{R}^{2}$ and $\alpha\in\mathbb{N}^{2}$,
where $\overline{D}_{\alpha}$ are positive constants depending only
on $\alpha$. Using the inverse Fourier transform we have
\[
p_{n,m}(2^{n}\lambda_{1},2^{m}\lambda_{2})e^{\lambda_{1}+\lambda_{2}}=\frac{1}{2\pi}\int\widehat{f}_{n,m}(\xi)e^{i\xi\cdot\lambda}d\xi_{1}d\xi_{2}.
\]
Therefore
\[
p_{n,m}(\lambda_{1},\lambda_{2})=\frac{1}{2\pi}\int\widehat{f}_{n,m}(\xi)e^{-\lambda_{1}\left(\frac{1}{2^{n}}-i\frac{\xi_{1}}{2^{n}}\right)}e^{-\lambda_{2}\left(\frac{1}{2^{m}}-i\frac{\xi_{2}}{2^{m}}\right)}d\xi_{1}d\xi_{2}. 
\]
Thus by spectral theory we have
\[
p_{n,m}(-\Delta)u_{1}(x_{1})u_{2}(x_{2})=\frac{1}{2\pi}\int\widehat{f}_{n,m}(\xi)e^{\Delta_{1}\left(\frac{1}{2^{n}}-i\frac{\xi_{1}}{2^{n}}\right)}u(x_{1})e^{\Delta_{2}\left(\frac{1}{2^{m}}-i\frac{\xi_{2}}{2^{m}}\right)}u(x_{2})d\xi.
\]
Hence the kernel of $p_{n,m}(-\Delta)$ is
\[
K_{n,m}(x,y)=\frac{1}{2\pi}\int\widehat{f}_{n,m}(\xi)h_{\frac{1}{2^{n}}-i\frac{\xi_{1}}{2^{n}}}^{(1)}(x_{1},y_{1})h_{\frac{1}{2^{m}}-i\frac{\xi_{2}}{2^{m}}}^{(2)}(x_{2},y_{2})d\xi_{1}d\xi_{2},
\]
for all $x=(x_{1},x_{2}),y=(y_{1},y_{2})\in X$, where $h^{(i)}$
is the heat kernel corresponding to $\Delta_{i}$, $i=1,2$. Using
inequalities \eqref{eq:heat_ker_est2} and \eqref{eq:est_Ffnm} we
obtain that
\begin{align*}
\left\vert K_{n,m}(x,y)\right\vert  & \le
C(d,j)2^{\frac{nd_{1}}{d_{1}+1}}\int\exp\left(-\frac{c\gamma R_{1}(x,y)^{(d_{1}+1)\gamma}2^{n\gamma_{1}}}{2(1+\xi_{1}^{2})^{\frac{\gamma_{1}+1}{2}}}\right)\frac{1}{(1+\xi_{1}^{2})^{j}}d\xi_{1}\\
 & \cdot  2^{\frac{md_{2}}{d_{2}+1}}\int\exp\left(-\frac{c\gamma_{2} R_{2}(x,y)^{(d_{2}+1)\gamma_{2}}2^{m\gamma_{2}}}{2(1+\xi_{2}^{2})^{\frac{\gamma_{2}+1}{2}}}\right)\frac{1}{(1+\xi_{2}^{2})^{k}}d\xi_{2},
\end{align*}
for all $j,k\ge1$. Lemma \eqref{lem:ineq} implies that (if we choose
$j$ and $k$ large enough)
\begin{align*}
\vert K_{p}(x,y)\vert & =  \left\vert \sum_{m\in\mathbb{Z}}\sum_{n\in\mathbb{Z}}K_{n,m}(x,y)\right\vert \le C(d)R_{1}(x_{1},y_{1})^{-d_{1}}R_{2}(x_{2},y_{2})^{-d_{2}}.
\end{align*}
Repeating the above proof for the function $q_{\beta_{1},\beta_{2}}(\lambda)=\lambda_{i}^{\beta_{1}}\lambda_{j}^{\beta_{2}}p(\lambda)$,
we obtain that the kernel $K_{p}(x,y)$ is smooth off the diagonal
and satisfies the following estimates
\[
\left\vert \Delta_{x,i}^{\beta_{1}}\Delta_{y,j}^{\beta_{2}}K_{p}(x,y)\right\vert \le C(d,\beta_{1},\beta_{2})R_{1}(x,y)^{-\beta_{1}(d_{1}+1)-d_{1}}R_{2}(x_{2},y_{2})^{-\beta_{2}(d_{2}+1)-d_{2}}.
\]
\end{proof}
\begin{cor}
Assume that $X$ is a product of fractafolds without boundary. If
$p\in S^{0}$ then $p(-\Delta)$ is a Calder\'on-Zygmund operator
and, thus, it extends to a bounded operator on $L^{q}(\mu)$ for all
$1<q<\infty$ and satisfies weak $1-1$ estimates.\end{cor}
\begin{proof}
Our theorem \ref{thm:smooth_ker_mul} and Theorem 6.1 of \cite{IR_2010}
imply that $p(-\Delta)$ is a Calder\'on-Zygmund operator.\end{proof}
\begin{rem}
The boundedness of $p(-\Delta)$ on $L^{q}(\mu)$ for all $1<q<\infty$
can be deduced also from the results of \cite{sikora-2008}. \end{rem}
\begin{cor}
If $p\in S^{m}$ then $p(-\Delta)$ is given by integration with respect
to a kernel that is smooth off the diagonal.\end{cor}
\begin{example}
Consider the Riesz potentials (\cite[page 591]{Str:TAM04})
\[
R_{i}u=\sum\frac{\lambda_{i}}{\lambda_{1}+\dots+\lambda_{N}}\proj u,\; i=1,\dots,N.
\]
Clearly $p_{i}(\lambda_{1},\dots,\lambda_{N})=\frac{\lambda_{i}}{\lambda_{!}+\dots+\lambda_{N}}$,
$i=1,\dots N$ are in $S^{0}$ and thus the Riesz potentials are bounded
on $L^{p}(\mu)$, for $1<p<\infty$, and are given by integration
with respect to smooth kernels.
\end{example}

\section{\label{sec:Elliptic-and-Hypoelliptic}Elliptic and Hypoelliptic operators}

In this section we specialize the set-up of section \ref{sec:products}
to the case when $X$ is the product of $N$ fractafolds without boundary.
If $X_{i}$ is not compact, we assume that $\Delta_{i}$ has pure
point spectrum. We still write $\lambda=(\lambda_{1},\dots,\lambda_{N})\in(0,\infty)^{N}$,
$x=(x_{1},\dots,x_{N})\in X$ and $\Delta=(\Delta_{1},\dots,\Delta_{N})$.
We begin by extending the definition of elliptic operators from \cite[Definition 8.2]{RogStr-2009_distr}
to elliptic pseudo-differential operators. We show that any pseudo-differential
operator satisfies the pseudo-local property and the converse inclusion
is also true for elliptic operators. In other words, we prove that
elliptic operators on fractals are hypoelliptic. These results give
a positive answer to some open questions posed in \cite{Str_JFA03,RogStr-2009_distr}.

We use the theory of distributions on fractals as developed in \cite{RogStr-2009_distr}
for our definitions and results. Most of the results, however, are
true for more general metric spaces that we consider in Section \ref{sec:background}
if one  replaces the word {}``distribution'' with {}``function''
in what follows.

Recall that the space of test functions $\D(X)$ consists of all smooth
functions with compact support. The space of distributions on $X$
is the dual space $\D^{\prime}(X)$ of $\D(X)$ with the weak-star
topology (\cite[Definition 4.1]{RogStr-2009_distr}). A distribution
$u$ is \emph{smooth} on the open set $\Omega\subset X$ (\cite[Definition 8.1]{RogStr-2009_distr})
if there exists a smooth function $f$ on $\Omega$ such that
\begin{equation}
\langle u,\phi\rangle=\int f\phi d\mu\label{eq:smooth_dist}
\end{equation}
for all $\phi\in\D(\Omega)$. If $\Omega_{u}$ is the maximal open
set on which $u$ is smooth, then the \emph{singular support} of $u$
is (\cite[Definition 8.2]{RogStr-2009_distr})
\[
\singsupp(u)=X\setminus\Omega_{u}.
\]
If $p\in S^{m}$ then one can extend the definition of $p(-\Delta)$
to the space of distributions via
\[
\langle p(-\Delta)u,\phi\rangle=\langle u,p(-\Delta)\phi\rangle,
\]
for all $\phi\in\D(X)$.
\begin{defn}
We say that $p(-\Delta)$ satisfies the \emph{pseudo-local }property
if
\[
\singsupp(p(-\Delta)u)\subseteq\singsupp(u).
\]
\end{defn}
\begin{thm}
\label{pro:pseudo-local}If $p$ is in $S^{m}$ then $p(-\Delta)$
satisfies the pseudo-local property.\end{thm}
\begin{proof}
Let $u\in\D^{\prime}(X)$ and let $K_{u}=\singsupp(u)$. Then
$u|_{\Omega}$ is smooth, where $\Omega=X\setminus K_{u}$. We need
to show that $Pu|_{\Omega}$ is also smooth. Let $K$ be be any finite
union of cells in $\Omega$ and let $\Omega^{\prime}$ the interior
of $K$. We identify $u|_{\Omega}$ with the smooth functions given
by \eqref{eq:smooth_dist}. Then, using \cite[Theorem 7.7]{RogStr-2009_distr},
we can find a smooth function $u_{1}$ such that $u_{1}|_{K}=u|_{K}$,
and $u_{1}|_{\Omega^{c}}=0$. We write then $u=u_{1}+u_{2}$, where
$u_{2}$ is supported on $K^{c}$. Then $Pu=Pu_{1}+Pu_{2}$. Since
$u_{1}$ is smooth it follows that $Pu_{1}$ is smooth. Therefore
we only need to prove that $Pu_{2}|_{\Omega^{\prime}}$ is smooth.
For this we are going to use the structure theorem for distributions
\cite[Theorem 5.10]{RogStr-2009_distr} (see also the comments following
Assumption 7.2 in the same paper). Then $u_{2}$ may be written as
a locally finite sum of the form $u_{2}=\sum\Delta^{m_{j}}\nu_{j}$
or $u_{2}=\sum\Delta^{m_{j}+1}f_{j}$, where the $\nu_{j}$ are Radon
measures and the $f_{j}$ are continuous functions with compact support.
Notice that the support of each $\nu_{j}$ is a subset of $K^{c}$
and, respectively, the support of each $f_{j}$ is a subset of $K^{c}$.

Assume first that $u_{2}=\sum\Delta^{m_{j}}\nu_{j}$ where $\nu_{j}$
are Radon measures. Let $\phi\in\D(\Omega^{\prime}).$ Then
\begin{eqnarray*}
\langle p(-\Delta)u_{2},\phi\rangle & = & \sum_{j}\int\Delta^{m_{j}}P\phi(x)d\nu_{j}(x)\\
 & = & \sum\int\Delta^{m_{j}}\int K_{p}(x,y)\phi(y)d\mu(y)d\nu_{j}(x)\\
 & = & \int\sum\int\Delta^{m_{j}}K_{p}(x,y)d\nu_{j}(x)\phi(y)d\mu(y),
\end{eqnarray*}
where $K_{p}$ is the kernel of $p(-\Delta)$ provided by Theorem
\ref{thm:smooth_ker_mul}. The above expression makes sense because
the support of $\phi$ is a subset of $\Omega^{\prime}$, while the
supports of $\nu_{j}$ are subsets of $K^{c}$. Since $K_{p}$ is
smooth it follows that the function $x\mapsto\sum\int\Delta^{m_{j}}K_{p}(x,y)d\nu_{j}(y)$
is smooth for all $x\in\Omega^{\prime}$. Thus $p(-\Delta)u_{2}|_{\Omega^{\prime}}$
is smooth. A similar proof shows that $p(-\Delta)u_{2}|_{\Omega^{\prime}}$
is smooth if $u_{2}=\sum\Delta^{m_{j}+1}f_{j}$. It follows that $p(-\Delta)u$
is smooth on $\Omega^{\prime}$. Since $K$ was arbitrarily chosen,
we have that $p(-\Delta)u$ is smooth on $\Omega$.
\end{proof}
We turn now to the study of elliptic pseudo-differential operators.
The first part of the following definition is from \cite[Definition 8.4]{RogStr-2009_distr}.
\begin{defn}
An \emph{elliptic polynomial} of degree $m$ on $\mathbb{R}^{N}$
is any polynomial $q$ of degree $m$ satisfying the property that
there exist two positive constants $c$ and $A$ such that
\begin{equation}
\vert q(\lambda)\vert\ge c\vert\lambda\vert^{\frac{m}{d+1}}\mbox{ for all }\lambda\in\mathbb{R}_{+}^{N}\text{ such that }\vert\lambda\vert\ge A.\label{eq:elliptic}
\end{equation}
An \emph{elliptic differential operator} on $X$ is an operator $q(-\Delta)$
for some elliptic polynomial $q(\lambda)$. More generally, an \emph{elliptic
pseudo-differential operator} of order $m$ on $X$ is a pseudo-differential
operator $p(-\Delta)$ whose symbol $p\in S^{m}$ satisfies \eqref{eq:elliptic}.\end{defn}
\begin{rem}
Notice that we require that condition \eqref{eq:elliptic} holds only
for nonnegative values of $\lambda_{1},\dots,\lambda_{N}$. In particular,
the polynomial $p(\lambda)=\lambda_{1}+\dots+\lambda_{N}$ is elliptic
according to our definition.\end{rem}
\begin{thm}
\label{thm:elliptic_hyppoelliptic}An elliptic pseudo-differential
operator $p(-\Delta)$ is hypoelliptic, mea-ning that
\[
\singsupp(p(-\Delta)u)=\singsupp(u).
\]
\end{thm}
\begin{proof}
Since $p$ is elliptic the set of its zeros is compact and $p(\lambda)\ne0$
for all $\lambda\in\mathbb{R}^{N}$ with $\vert\lambda\vert\ge A$.
Let $\xi\in C^{\infty}$ be a cut-off function which is $0$ on a
neighborhood of the zeros of $p$ and is $1$ for $\vert\lambda\vert\ge A$.
Define $p_{1}(\lambda)=\xi(\lambda)/p(\lambda)$, so that $p_{1}(\lambda)$
is zero on a neighborhood of the zeros of $p$. Then $p(\lambda)p_{1}(\lambda)=1$
if $\vert\lambda\vert\ge A$. Moreover $p_{1}\in S^{-m}$, since for
$\vert\lambda\vert\ge A$ we have that (see, for example, \cite[Proof of Theorem 1.3]{Tay_PSDO81})
\[
\frac{\partial^{\alpha}}{\partial\lambda^{\alpha}}p_{1}(\lambda)=\sum_{\alpha^{1}+\alpha^{2}+\dots+\alpha^{\mu}=\alpha}\frac{\partial^{\alpha^{1}}}{\partial\lambda^{\alpha^{1}}}p(\lambda)\dots\frac{\partial^{\alpha^{\mu}}}{\partial\lambda^{\alpha^{\mu}}}p(\lambda)\cdot p(\lambda)^{1-\mu}.
\]
Using Lemma \ref{lem:Symboliccalculus} and Theorem \ref{thm:smooth_ker_mul}
we have that
\begin{align*}
p_{1}(-\Delta)p(-\Delta)u & =u+Ru,
\end{align*}
where $R\in S^{-\infty}$ has a smooth compactly supported symbol
$r(\lambda)$. Hence $R$ is infinitely smoothing. It follows that
if $\Omega$ is an open set such that $p(-\Delta)u|_{\Omega}$ is
smooth then $u|_{\Omega}=p_{1}(-\Delta)p(-\Delta)u|_{\Omega}-Ru|_{\Omega}$
is smooth. Therefore $\singsupp(u)\subseteq\singsupp(p(-\Delta)u)$.
Since the converse is provided by Theorem \ref{pro:pseudo-local},
it follows that $p(-\Delta)$ is hypoelliptic.
\end{proof}

\subsection{\label{sub:Quasielliptic-operators}Quasielliptic operators}

It is known (\cite{BoStr_IUMJ07,DrStr_IJM09}) that for some fractals
the set of ratios of eigenvalues of the Laplacian has gaps. That is,
there are $\alpha<\beta$ such that $\lambda/\lambda^{\prime}\notin(\alpha,\beta)$
for all $\lambda\ne\lambda^{\prime}$ in the spectrum of $-\Delta$.
Consider the product $X=X_{1}\times X_{2}$ of two copies of the same
fractafold without boundary that have a spectral gap $(\alpha,\beta)$.
Let $\Delta_{1}$ and $\Delta_{2}$ be the Laplacian on the fractafold
viewed as acting on $X_{1}$ and, respectively, on $X_{2}$. As before,
we write $\Delta=(\Delta_{1},\Delta_{2})$. For a real number $a$
in $(\alpha,\beta)$ and $\varepsilon>0$ define the cone
\[
\Gamma_{a,\varepsilon}=\{(\lambda_{1},\lambda_{2})\in\mathbb{R}_{+}^{2}\::\:\vert\lambda_{1}-a\lambda_{2}\vert<\varepsilon\lambda_{2}\}.
\]
If $a\in(\alpha,\beta)$ then there is $\varepsilon>0$ such that
$\Gamma_{a,\varepsilon}$ does not contain any pair of eigenvalues
of $-\Delta$. Let $\mathfrak{G}$ be the collection of such cones.
Denote by $\mathcal{G}$ the collection of cones in $\mathbb{R}_{+}^{2}$
that do not intersect any cone in $\mathfrak{G}$. Thus any pair of
eigenvalues of $-\Delta$ lies in some cone in $\mathcal{G}$. Moreover,
any pseudo-differential operator will {}``live'' on $\mathcal{G}$,
as we will make it precise in Lemma \ref{lem:quasi_is_el}.
\begin{defn}
We say that an $m$-symbol $p$ is \emph{quasielliptic} if it satisfies
\eqref{eq:elliptic} for all $\lambda\in\Gamma\text{ and }\Gamma\in\mathcal{G}$,
with the constant $A$ depending on $\Gamma$. A pseudo-differential
operator $p(-\Delta)$ is called a quasielliptic operator if $p$
is a quasielliptic symbol.\end{defn}
\begin{example}
If $a\in(\alpha,\beta)$ then the operator $\Delta_{1}-a\Delta_{2}$
is a quasielliptic operator (\cite{BoStr_IUMJ07,sikora-2008}). An
intriguing consequence of the existence of spectral gaps is that,
if $a\in(\alpha,\beta)$ then the operator $1/(\Delta_{1}-a\Delta_{2})$
is bounded on $L^{p}(\mu)$ for all $1<p<\infty$ \cite{sikora-2008}.
The next proposition show that the quasielliptic operators are, in
fact, elliptic operators. \end{example}
\begin{prop}
\label{pro:quasielliptic=00003Delliptic}A quasielliptic pseudo-differential
operator $p(-\Delta)$ is equal to an elliptic pseudo-differential
operator.
\end{prop}
This proposition follows immediately from the following lemma which
says that a quasielliptic symbol equals an elliptic symbol outside
a cone contained in the spectral gaps. Note that even if $p(-\Delta)$
is a differential operator, the elliptic operator is only a pseudo-differential
operator.
\begin{lem}
\label{lem:quasi_is_el}Let $p$ be a quasielliptic $m$-symbol. Then
there is an elliptic $m$-symbol $\tilde{p}$ such that $p(\lambda)=\tilde{p}(\lambda)$
for all $\lambda\in\Gamma$ and all $\Gamma\in\mathcal{G}$.\end{lem}
\begin{proof}
Let $a\in(\alpha,\beta)$ and $\varepsilon>0$ such that the cone
$\Gamma_{a,\varepsilon}$ is a maximal cone in $\mathfrak{G}$ and
let $\Gamma=\mathbb{R}_{+}^{2}\setminus\Gamma_{a,\varepsilon}$. Then
$\Gamma$ is a union of two simply connected cones. Since $p$ is
quasielliptic, there are $c>0$ and $A>0$ such that $\vert p(\lambda)\vert\ge c\vert\lambda\vert$
for all $\lambda\in\Gamma$ with $\vert\lambda\vert>A$. Since $\Gamma_{A}:=\{\lambda\in\Gamma\,:\,\vert\lambda\vert>A\}$
is a union of two disjoint simply connected sets, we can define the
function $q(\lambda)=\log p(\lambda)$. Let $r(\lambda)$ and $s(\lambda)$
be the real part and the imaginary part, respectively, of $q(\lambda)$.
By hypothesis, $r(\lambda)\ge\frac{m}{d+1}\log\vert\lambda\vert$
for all $\lambda\in\Gamma_{A}$. Let $\tilde{r}(\lambda)$ and $\tilde{s}(\lambda)$
be two smooth extensions of $r(\lambda)$ and $s(\lambda)$ such that
$\tilde{r}(\lambda)\ge\frac{m}{d+1}\log\vert\lambda\vert$ for all
$\lambda\in\mathbb{R}_{+}^{2}$ with $\vert\lambda\vert>A$. Let $\tilde{p}(\lambda)=\exp(\tilde{r}(\lambda)+i\tilde{s}(\lambda))$
and extend it so that $\tilde{p}(\lambda)=p(\lambda)$ for all $\lambda\in\Gamma$.
Then $\tilde{p}$ is the desired elliptic $m$-symbol.
\end{proof}

\begin{proof}
[Proof of Proposition \ref{pro:quasielliptic=00003Delliptic}] Let
$q$ be an elliptic symbol $q$ such that $p(\lambda)=q(\lambda)$
for all $\lambda\in\Gamma$ and all $\Gamma\in\mathcal{G}$. Then
$q(-\Delta)$ is an elliptic pseudo-differential operator. Moreover, since $p(\lambda)=q(\lambda)$
for all $\lambda\in\Lambda_{1}\times\Lambda_{2}$ (where $\Lambda_{i}$
is the spectrum of $-\Delta_{i})$, it follows from the definition
that $p(-\Delta)=q(-\Delta)$. Therefore $p(-\Delta)$ is equal to
an elliptic pseudo-differential operator.
\end{proof}
As an immediate consequence of the above Proposition and Theorem \ref{thm:elliptic_hyppoelliptic}
we obtain the following corollary, which answers an open question
posed in \cite{RogStr-2009_distr}.
\begin{cor}
A quasielliptic operator is hypoelliptic.
\end{cor}

\section{\label{sec:rho_Hormander}$\rho$-type symbols and pseudodifferential
operators}

In this section we define and study pseudo-differential operators
for which the derivatives of the symbols have a slower rate of decay.
These operators still have a kernel that is smooth off the diagonal,
even though they might not be bounded on $L^{q}$-spaces, as the well
known examples in the Euclidian setting show. As an application, we
study the so called H\"ormander type hypoelliptic operators and 
prove that they are hypoelliptic in this more general setting.  In the following we use
the set-up of Section \ref{sec:Symbols-and-pseudo-differential}.
The definitions and main theorem of this section can be extended without
difficulty to the product setting so we omit the details (see Remark \ref{Note-product}). We will use the fact that the
equivalent of Theorem \ref{thm:smooth_kernel-1} holds for products
of measure metric spaces in Subsection \ref{sub:Hrmander-type-hypoelliptic}.
\begin{defn}
For fixed $m\in\mathbb{R}$ and $0\le\rho\le1$ define the symbol
class $S_{\rho}^{m}$ to be the set of $p\in C^{\infty}((0,\infty))$
with the property that for any $k\ge0$ there is $C_{k}(\rho)>0$
such that
\[
\left\vert \bigl(\lambda^{\rho}\frac{d}{d\lambda}\bigr)^{k}p(\lambda)\right\vert \le C_{k}(\rho)(1+\lambda)^{\frac{m}{d+1}}
\]
for all $\lambda>0$.
\end{defn}

\begin{defn}
Define the operator class $\Psi DO_{\rho}^{m}$ by
\[
p(-\Delta)u=\int_{0}^{\infty}p(\lambda)P(\lambda)u
\]
for $p\in S_{\rho}^{m}$ and $u\in\mathcal{D}$.\end{defn}
\begin{thm}
\label{thm:smooth_kernel-1}Let $\frac{1}{\gamma+1}<\rho\le1$ and let $p:(0,\infty)\to\mathbb{C}$
be an $S_{\rho}^{0}$-symbol, that is, for each $k\ge0$ there is
$C_{k}(\rho)>0$ such that
\begin{equation}
\left\vert \lambda^{\rho k}\frac{\partial^{k}}{\partial\lambda^{k}}p(\lambda)\right\vert \le C_{k}(\rho).\label{eq:symbol-1}
\end{equation}
Then $p(-\Delta)$ has a kernel $K(x,y)$ that is smooth off the diagonal
of $X\times X$.\end{thm}
\begin{proof}
The proof of this theorem follows in large the proof of Theorem \ref{thm:smooth_kernel}.
We will present the main steps of the proof skipping the computations
that are very similar. Let $\delta$ be the function from the Littlewood-Paley
decomposition and let $p_{n}(\lambda)=p(\lambda)\delta(2^{-n}\lambda)$.
Then $\supp p_{n}\subseteq[2^{n-1},2^{n+1}]$ and
\[
\left\vert \frac{d^{k}}{d\lambda^{k}}p_{n}(\lambda)\right\vert \lesssim2^{-n\rho k}\;\text{for all}\; k\ge0.
\]
Define $\widetilde{p}_{n}(\lambda)=p_{n}(2^{n\rho}\lambda)$. Then
$\supp\widetilde{p}_{n}\subseteq[2^{n(1-\rho)-1},2^{n(1-\rho)+1}]$
and for each $k\ge0$ there is a constant $C_{k}(\rho)>0$ such that
\[
\left\vert \frac{d^{k}\widetilde{p}_{n}(\lambda)}{d\lambda^{k}}\right\vert \le C_{k}(\gamma).
\]
Let $f_{n}(\lambda)=\widetilde{p}_{n}(\lambda)e^{\frac{\lambda}{2^{n(1-\rho)}}}$.
Since $e^{\frac{\lambda}{2^{n(1-\rho)}}}\le e^{2}$ for all $\lambda\in[2^{n(1-\rho)-1},2^{n(1-\rho)+1}]$,
it follows that for each $k\ge 0$ there is a constant
$A_k(\gamma)>0$ independent on $n$ such that
\[
\left\vert \frac{d^{k}f_{n}(\lambda)}{d\lambda^{k}}\right\vert
\le A_k(\gamma).
\]
Therefore
\[
\vert\xi^{k}\widehat{f}_{n}(\xi)\vert\lesssim 1\;\text{for all}\; k\ge0.
\]
In particular we have that for each $k\ge 0$ there is a constant
$\overline{D}_k(\gamma)$ independent on $n$ such that
\begin{equation}
\vert\widehat{f}_{n}(2^{n(\rho-1)}\xi)\vert\le\frac{\overline{D}_k(\gamma)2^{2nk(1-\rho)}}{(1+\xi^{2})^{k}}\;\text{for 
  all}\; k\ge0.\label{eq:fourier_f_n} 
\end{equation}
Using the Fourier inversion formula we have that
\[
p_{n}(2^{n\rho}\lambda)e^{\frac{2^{n\rho}\lambda}{2^{n}}}=\frac{1}{2\pi}\int\widehat{f}_{n}(\xi)e^{i\lambda\xi}d\xi.
\]
Therefore
\begin{eqnarray*}
p_{n}(\lambda) & = & \frac{1}{2\pi}\int\widehat{f}_{n}(\xi)e^{-\lambda\left(\frac{1}{2^{n}}-i\frac{\xi}{2^{n\rho}}\right)}d\xi\\
 & = & \frac{1}{2\pi}\int\widehat{f}_{n}(\xi)e^{-\frac{\lambda}{2^{n}}\left(1-i\frac{\xi}{2^{n(\rho-1)}}\right)}d\xi,
\end{eqnarray*}
which, using the substitution $u=\xi/2^{n(\rho-1)}$ becomes
\[
p_{n}(\lambda)=\frac{2^{n(\rho-1)}}{2\pi}\int\widehat{f}_{n}(2^{n(1-\rho)}\xi)e^{-\frac{\lambda}{2^{n}}(1-i\xi)}d\xi.
\]
Thus the kernel of $p_{n}(-\Delta)$ equals
\[
K_{n}(x,y)=\frac{2^{n(\rho-1)}}{2\pi}\int\widehat{f}_{n}(2^{n(1-\rho)}\xi)h_{\frac{1}{2^{n}}-i\frac{\xi}{2^{n}}}(x,y)d\xi.
\]
Using the estimates \eqref{eq:heat_ker_est2} and
\eqref{eq:fourier_f_n} we obtain that
\begin{align*}
\vert K_n(x,y)\vert &\lesssim 2^{\frac{n}{d+1}(d+(1-\rho)(d+1)(2j-1))}
\\
& \cdot  \int\exp\left(-\frac{c\gamma}{2}R(x,y)^{(d+1)\gamma}2^{n\gamma}(1+\xi^{2})^{-\frac{\gamma+1}{2}}\right)\frac{1}{(1+\xi^{2})^{j}}d\xi,
\end{align*}
for all $j\ge 0$.
Lemma \ref{lem:ineq} implies that if we pick $j$ such that
\[
j \ge \frac{(d+(1-\rho)(d+1)(2j-1)(\gamma+1)}{2\gamma(d+1)}+\frac{1}{2}
\]
then $K(x,y)=\sum_{n}K_{n}(x,y)$
converges. We can pick a positive $j$ in the equation above provided
that $\rho>\frac{1}{\gamma+1}$. Then it follows that the kernel $K$ satisfies the estimates
\[
\vert K(x,y)\vert\lesssim R(x,y)^{-\frac{d\gamma}{\rho(\gamma+1)-1}}.
\]
A similar argument for $q(\lambda)=\lambda^{j}p(\lambda)$, $j\ge1$,
shows that $K$ is smooth off the diagonal.\end{proof}
\begin{cor}
If $\frac{1}{\gamma+1}<\rho\le 1$ and $p\in S_{\rho}^{m}$ then the kernel of $p(-\Delta)$ is smooth
off the diagonal.\end{cor}
\begin{proof}
The proof is very similar with that of Corollary \eqref{cor:smooth_kernel_m}.\end{proof}
\begin{rem}
\label{Note-product}Note that for $\rho<1$ the operators $p(-\Delta)$
are \emph{not} Calder\'on-Zygmund operators and they might not extend
to $L^{q}$ for all values of $q$. Also, one can easily extend the
definition to product of fractals and adapt the proof of Theorem \ref{thm:smooth_ker_mul}
to show that the kernel of a pseudo-differential operator of type
$\rho$ with $\frac{1}{\gamma+1}<\rho\le1$ is smooth off the diagonal.
\end{rem}

\subsection{\label{sub:Hrmander-type-hypoelliptic}H\"ormander type hypoelliptic
operators}

Let $X$ be a product of $N$ measure metric space as in the setting
of Section \ref{sec:products}. Recall that we write $\Delta=(\Delta_{1},\dots,\Delta_{N})$.
\begin{defn}
(\cite{hormander1985analysis}) We say that a smooth map $p:(0,\infty)^{N}\to\mathbb{C}$
is a H\"ormander type hypoelliptic symbol if there are $\varepsilon>0$
and $A>0$ such that
\begin{equation}
\left\vert \frac{\frac{\partial^{\alpha}}{\partial\lambda^{\alpha}}p(\lambda)}{p(\lambda)}\right\vert \le c_{\alpha}\vert\lambda\vert^{-\varepsilon\vert\alpha\vert}\text{ for }\vert\lambda\vert\ge A,\label{eq:Hormander}
\end{equation}
where $c_{\alpha}$ are postive constants for all $\alpha\in\mathbb{N}^{N}$. \end{defn}
\begin{thm}
Suppose that $p$ is a H\"ormander type hypoelliptic symbol with $\varepsilon>\frac{1}{\gamma+1}$. Then
$p(-\Delta)$ is hypoelliptic.\end{thm}
\begin{proof}
We know from \eqref{eq:Hormander} that the set of zeros of $p$ is
compact. Let $\varphi$ be a smooth function that is $0$ on a neighborhood
of the zeros of $p$ and $\varphi(\lambda)=1$ if $\vert\lambda\vert\ge A$.
Set $q(\lambda)=\varphi(\lambda)p(\lambda)^{-1}$. Then
\[
\frac{\partial^{\alpha}}{\partial\lambda^{\alpha}}q(\lambda)=\sum_{\alpha_{1}+\cdots+\alpha_{\mu}}\frac{\partial^{\alpha_{1}}}{\partial\lambda^{\alpha_{1}}}p(\lambda)\cdots\frac{\partial^{\alpha_{\mu}}}{\partial\lambda^{\alpha_{\mu}}}p(\lambda)\cdot p(\lambda)^{1-\mu},
\]
if $\vert\lambda\vert\ge A$. It follows that $q\in S_{\varepsilon}^{0}$.
Since $q(\lambda)p(\lambda)=1$ if $\vert\lambda\vert\ge A$ we have
that
\[
q(-\Delta)p(-\Delta)=u+Ru,
\]
where $R$ is an infinitely smoothing operator. Theorem \ref{thm:smooth_kernel-1}
implies that $p(-\Delta)$ is hypoelliptic.
\end{proof}

\section{\label{sec:wavefront}Wavefront set and microlocal analysis}

In this section we initiate the study of the microlocal analysis on
fractals. Namely, we define the wavefront sets on fractals and provide
a few concrete examples. As in classical harmonic analysis, the
wavefront sets contain information about the singularities of distributions.
We show that a pseudo-differential operator reduces, in general, the
wavefront set. If the operator is elliptic, then the wavefront set
remains unchanged. This is just a beginning of this study. We expect
that the wavefront set will play an important role in the future,
but we can not expect results analogous to classical analysis on propagation
of singularities for space-time equations being controlled by the
wavefront set, as there are localized eigenfunctions which prevent
singularities from propagating at all. We continue to assume the set-up
and notations of Section \ref{sec:Elliptic-and-Hypoelliptic}. Namely,
$X$ is the product of $N$ fractafolds without boundary, $\mu$ is
the product measure on $X$, and $\Delta=(\Delta_{1},\dots,\Delta_{N})$.
In addition, we assume in this section that each fractafold $X_{i}$
is \emph{compact. }
\begin{defn}
Let $\Gamma$ denote an open cone in $\mathbb{R}_{+}^{N}$ and $\Omega$
an open set in $X$.  We use $\varphi_{\alpha_{k}}$ to denote $L^{2}$ normalized eigenfunctions corresponding
to eigenvalues $\lambda_{\alpha_{k}}$, and set $\lambda_{\alpha}=\lambda_{\alpha_{1}}+\dots+\lambda_{\alpha_{N}}$.
A distribution $u$ is defined to be $C^{\infty}$ in
$\Omega\times\Gamma$ if it can be written on $\Omega$ as a linear combination of eigenfunctions with coefficients having faster than polynomial decay over the eigenvalues in $\Gamma$.   More precisely, if there is a sequence $b_{n}$ and a function $v$ with $v|_{\Omega}=u$ that has the form
\begin{equation}
v=\sum_{\alpha}c_{\alpha}\varphi_{\alpha_{1}}\otimes\varphi_{\alpha_{2}}\otimes\dots\otimes\varphi_{\alpha_{N}},\label{eq:Cinfty}
\end{equation}
for values $c_{\alpha}$ such that $\vert c_{\alpha}\vert\le b_{n}(1+\lambda_{\alpha})^{-n/(d+1)}$ 
for all $n$ and all $\{\lambda_{\alpha_{1}},\dots,\lambda_{\alpha_{N}}\}\in\Gamma$. We define the \emph{wavefront set }of
$u$, \global\long\def\WF{\operatorname{WF}}
 $\WF(u)$, to be the complement of the union of all sets where $u$
is $C^{\infty}$.\end{defn}
\begin{rem}
If $u$ is a smooth function on $X$ then $\WF(u)$ is empty. More
generally, $\singsupp u$ is the projection of $\WF(u)$ onto $X$.
\end{rem}

\begin{rem}
If the fractals $X_{i}$ have gaps in the set of ratios of eigenvalues,
then there are special cones $\Gamma$ for which every $u$ is $C^{\infty}$
on $X\times\Gamma$$ $ because $\Gamma$ contains no eigenvalues
(see Subsection \ref{sub:Quasielliptic-operators}).\end{rem}
\begin{prop}
If $p\in S^{m}$ then $\WF p(-\Delta)u\subseteq\WF(u)$. If in addition
$p(-\Delta)$ is elliptic then $\WF(p(-\Delta)u)=\WF(u)$.\end{prop}
\begin{proof}
Let $u$ be a distribution on $X$ and suppose that $\Omega$ is an
open subset of $X$ and $\Gamma$ is an open cone in $\mathbb{R}_{+}^{N}$
such that $u$ is $C^{\infty}$ in $\Omega\times\Gamma$. Let $v$
be as in \eqref{eq:Cinfty}. Then $Pv|_{\Omega}=Pu$ and
\[
Pv=\sum_{\alpha}c_{\alpha}p(\lambda_{\alpha_{1}},\dots,\lambda_{\alpha_{N}})\varphi_{\alpha_{1}}\otimes\dots\otimes\varphi_{\alpha_{N}},
\]
with $\vert c_{\alpha}p(\lambda_{\alpha_{1}},\dots,\lambda_{\alpha_{N}})\vert\le b_{n}(1+\lambda_{\alpha})^{-(n-m)/(d+1)}$
for all $\{\lambda_{\alpha_{1}},\dots,\lambda_{\alpha_{N}}\}\in\Gamma$.
Thus $Pu$ is $C^{\infty}$ in $\Omega\times\Gamma$ and, since $\Omega$
and $\Gamma$ were arbitrary, $\WF(Pu)\subseteq\WF(u)$.

If $P$ is elliptic then we can write, as in the proof of Theorem
\ref{thm:elliptic_hyppoelliptic}, $u=P_{1}Pu+Ru$, where $P_{1}$
is a pseudo-differential operator of order $-m$ and $R$ is a smoothing
operator. The conclusion follows immediately from this.\end{proof}
\begin{example}
Let $N=2$ and let $u=u_{1}\otimes u_{2}$ be a tensor product of distributions.
From the definition $u$ is $C^{\infty}$ on $\bigl((\singsupp u_{1})^{c}\times(\singsupp u_{2})^{c}\bigr)\times\mathbb{R}_{2}^{+}$,
thus $\WF(u)$ is a subset of the complement of this set. It is also
easy to check that
\[
\bigl(\singsupp(u_{1})\times\singsupp(u_{2})\bigr)\times\mathbb{R}_{2}^{+}\subseteq\WF(u).
\]
Let $x_{1}\in X_{1}$ and $x_{2}\in X_{2}$ and suppose that $u_{1}$
is smooth near $x_{1}$ but $x_{2}\in\singsupp(u_{2})$. Then
there is a smooth function $v_{1}$ on $X_{1}$ and a neighborhood
$\Omega_{1}$ of $x_{1}$ such that $u_{1}=v_{1}$ on $\Omega_{1}$.
Moreover we can write $v_{1}=\sum_{j}c_{j}\varphi_{j}$ with the Fourier
coefficients $c_{j}$ having faster decay then any polynomial (\cite[Theorem 3.5]{RogStr-2009_distr}).
Therefore for any $n\in\mathbb{N}$ there is $b_{n,j}>0$ such that
\[
\vert c_{j}\vert\le b_{n,j}(1+\lambda_{j})^{-n}.
\]
Also $u_{2}=\sum_{k}c_{k}^{\prime}\varphi_{k}$, where the coefficients
$c_{k}^{\prime}=\langle u_{2},\varphi_{k}\rangle$ have at most polynomial
growth (\cite[Lemma 4.4]{RogStr-2009_distr}). Thus for every $k\in\mathbb{N}$
there is $m_{k}\in\mathbb{N}$ and $b_{k}^{\prime}>0$ such that
\[
\vert c_{k}^{\prime}\vert\le b_{k}^{\prime}(1+\lambda_{k})^{m_{k}}.
\]
Define
\[
v=v_{1}\otimes u_{2}=\sum c_{j}c_{k}^{\prime}\varphi_{j}\otimes\varphi_{k}
\]
and let $\Gamma$ be a cone in $\mathbb{R}_{+}^{2}\setminus\{y-\mbox{axis}\}$.
Then there is $M>0$ such that for all $(\lambda^{\prime},\lambda^{\prime\prime})\in\Gamma$
we have that $\vert\lambda^{\prime\prime}\vert\le M\vert\lambda^{\prime}\vert$.
Therefore
\begin{eqnarray*}
\vert c_{j}c_{k}^{\prime}\vert & \le & b_{n,j}(1+\lambda_{j})^{-n}b_{k}^{\prime}(1+\lambda_{k})^{m_{k}}\\
 & \le & M^{m_{k}}b_{n,j}b_{k}^{\prime}(1+\lambda_{j})^{-n+m_{k}}\le C(M,j,k)(1+\lambda_{j}+\lambda_{k})^{-n+m_{k}},
\end{eqnarray*}
and thus $u$ is $C^{\infty}$ on $(\Omega_{1}\times X_{2})\times\Gamma$.

Similarly, one can show that if $x_{1}\in\singsupp u_{1}$ and
$u_{2}$ is smooth near $x_{2}$, then $u$ is $C^{\infty}$ on $(X_{1}\times(\singsupp(u_{2}))^{c})\times\Gamma$
for all cones $\Gamma$ in $\mathbb{R}_{+}^{2}\setminus\{x-\mbox{axis}\}$.
Therefore
\begin{eqnarray*}
\WF(u) & = & \bigl(\singsupp(u_{1})\times\singsupp(u_{2})\bigr)\times\mathbb{R}_{+}^{2}\\
 & \bigcup & \bigl((\singsupp(u_{1}))^{c}\times\singsupp(u_{2})\bigr)\times\{y-\mbox{axis}\}\\
 & \bigcup & \bigl(\singsupp(u_{1})\times(\singsupp(u_{2}))^{c}\bigr)\times\{x-\mbox{axis}\}.
\end{eqnarray*}

\end{example}

\begin{example}
We continue to assume that $N=2$. Let $x=(x_{1},x_{2})\in X$ and
let $\{(j,k)\}$ be a sequence such that the maps $\varphi_{j}\otimes\varphi_{k}$
are supported in a decreasing sequence of neighborhoods of $x$. Let $\{c_{jk}\}$
be any sequence of real numbers and define $u=\sum c_{jk}\varphi_{j}\otimes\varphi_{k}$.
Let $\Omega$ be an open subset of $X$ such that there is an open
neighborhood $V$ of $x$ so that $\Omega\bigcap V=\emptyset$. Then
there are only a finite number of indices $(j,k)$ with the property
that $\supp\varphi_{j}\otimes\varphi_{k}\bigcap\Omega\ne\emptyset$.
It follows that $u$ is $C^{\infty}$ on $\Omega\times\mathbb{R}_{+}^{2}$.
Since $\Omega$ was arbitrary, it follows that $\WF(u)\subseteq\{x\}\times\mathbb{R}_{+}^{2}$.
Let $\Gamma$ be a cone in $\mathbb{R}_{+}^{2}$. Define $c_{jk}=1$
if $(\lambda_{j},\lambda_{k})\in\Gamma$ and set $c_{jk}$ to be zero
otherwise. Then for $u=\sum c_{jk}\varphi_{j}\otimes\varphi_{k}$
we have that $\WF(u)=\{x\}\times\Gamma$.

More generally, let $K$ be any compact subset of $X^{2}$. Let $\{(j,k)\}$
be a sequence such that $\varphi_{j}\otimes\varphi_{k}$ is supported
in a decreasing sequence of neighborhoods of $K$. Then if $u=\sum c_{jk}\varphi_{j}\otimes\varphi_{k}$
with $\{c_{jk}\}$ an arbitrary sequence, we have that $\WF(u)\subseteq K\times\mathbb{R}_{+}^{2}$.
If $\Gamma$ is a cone in $\mathbb{R}_{+}^{2}$ and we define $c_{jk}$
as before, we obtain that $\WF(u)=K\times\mathbb{R}_{+}^{2}$.
\end{example}

\section{Pseudo-differential operators with variable coefficients }

The symbols of the operators studied in the previous section are independent
of the $x$ variable. These symbols are also known as \emph{constant}
coefficient symbols, and the corresponding operators are the constant
coefficient operators. In this section we extend our study to operators
whose symbols depend on both the $\lambda$ and $x$ variables. We
call them pseudo-differential operators with \emph{variable coefficients.}
The main difficulty in studying these operators is the fact that for
a large class of fractals the domain of the Laplacian is not closed
under multiplication \cite{BassStrTep_JFA99}. This implies in our
case that the symbolic calculus is not valid for the pseudo-differential
operators with variable coefficients. Namely, the product of two symbols
in the sense of Definition \ref{def:symbol_var_coeff} is no longer
smooth and it can not be a symbol of a pseudo-differential operator.
Another consequence of this fact is that the kernel of these operators
cannot be smooth. Nevertheless, we conjecture that these are Calder\'on-Zygmund
operators in the sense of \cite{Ste93,IR_2010}. We prove, in fact,
that the kernels of these operators are continuous off diagonal and
satisfy the correct decay off the diagonal. Moreover, we show that
these operators are bounded on $L^{q}(\mu)$ for all $1<q<\infty$.

We can define the operators if $X$ is either a compact fractafold
without boundary based on a self-similar nested fractal $K$, or an
infinite blow-up of $K$ without boundary. However we prove the
main properties in the case when $X$ is compact. We continue to write
$\D$ for the set of finite linear combinations of eigenvalues of
$-\Delta$ with compact support and $\Lambda$ for the spectrum of
$-\Delta$. We also assume that the heat kernel satisfies the upper estimate \eqref{eq:heat_estimates}.
\begin{defn}
\label{def:symbol_var_coeff}For $m\in\mathbb{R}$ we define the symbol
class $S^{m}$ to consist of the smooth functions $p:X\times(0,\infty)\to\mathbb{C}$
such that for each $k\in\mathbb{N}$ and $j\in\mathbb{N}$ there is
a positive constant $C_{j,k}$ such that
\begin{equation}
\left\vert \left(\lambda\frac{\partial}{\partial\lambda}\right)^{k}\Delta_{x}^{j}p(x,\lambda)\right\vert \le C_{j,k}(1+\lambda)^{\frac{m}{d+1}}.\label{eq:def_symbol_varcoef}
\end{equation}

\end{defn}
For $\lambda\in\Lambda$ we set $P_{\lambda}$ to be the corresponding
spectral projection.
\begin{defn}
We define the operator class of pseudo-differential operators with
variable coefficients $\Psi DO_{m}$ by
\[
p(x,-\Delta)u(x)=\sum_{\lambda\in\Lambda}\int p(x,\lambda)P_{\lambda}(x,y)u(y)dy
\]
for $p\in S^{m}$ and $u\in\mathcal{D}$.
\end{defn}
In the remainder of this section we will assume that $X$ is a \emph{compact}
fractafold without boundary. Recall that the set of eigenvalues is
countable and the only accumulation point is $\infty$. Let $\Lambda=\{\lambda_{1}\le\lambda_{2}\le\dots\}$
be the set of eigenvalues of $-\Delta$ in increasing order and with
repeated multiplicity. Let $\{\phi_{k}\}_{k\in\mathbb{N}}$ be an
orthonormal basis of $L^{2}(\mu)$ consisting of compactly supported
eigenfunctions.
\begin{thm}
Suppose that $X$ is a compact fractafold with no boundary and $p\in S^{0}$.
Then the operator $Tu(x)=p(x,-\Delta)u(x)$ extends to a bounded operator
on $L^{2}(\mu)$.\end{thm}
\begin{proof}
Define
\[
\widetilde{T}u(x,z)=\sum_{\lambda\in\Lambda}\int p(z,\lambda)P_{\lambda}(x,y)u(y)dy.
\]
Then
\[
\vert Tu(x)\vert=\vert\widetilde{T}u(x,x)\vert\le\sup_{z}\vert\widetilde{T}u(z,x)\vert\le\Vert(I-\Delta_{z})^{\frac{d}{2(d+1)}}\widetilde{T}u(\cdot,x)\Vert_{2},
\]
where the last inequality is the Sobolev inequality. So $\Vert Tu\Vert_{2}^{2}\le\Vert(I-\Delta_{z})\widetilde{T}u\Vert_{2}^{2}$.
If $p\in S^{0}$ then $\Delta_{z}p\in L^{\infty}$ because $X$ is
compact. Then the result follows. \end{proof}
\begin{rem}
The method of proof for the above theorem is adapted from \cite{Str_TAMS69}.
\end{rem}
Let $K(x,y)$ be the distributional kernel of $T$. Recall that we
can not expect that $K$ is smooth off the diagonal, because the domain
of the Laplacian is not, in general, closed under multiplication.
The best result one can expect is that $K$ is continuous off diagonal.
We prove that this is true in the following theorem.
\begin{thm}
\label{thm:continuous_ker}Suppose that $X$ is a compact fractafold
with no boundary and $p\in S^{0}$. Then the operator $Tu(x)=p(x,-\Delta)u(x)$
is given by integration with a kernel that is continuous off the diagonal
and satisfies the estimate
\begin{equation}
\vert K(x,y)\vert\lesssim R(x,y)^{-d}.\label{eq:est_ker_var_coeff}
\end{equation}
\end{thm}
\begin{proof}
Since the map $x\mapsto p(x,\lambda)$ is smooth for all $\lambda\in\mathbb{R}_{+}$,
it follows that we can write
\begin{equation}
p(x,\lambda)=\sum_{k=0}^{\infty}m_{k}(\lambda)\phi_{k}(x),\label{eq:p_sumofconstant}
\end{equation}
where $m_{k}(\lambda)=\langle p(\cdot,\lambda),\phi_{k}\rangle$,
for all $\lambda\in\mathbb{R}_{+}$. We claim that $\widetilde{m}_{k,n}(\lambda):=\lambda_{k}^{n}m_{k}(\lambda)$
is a $0$-symbol with \emph{constant} coefficients, for all $k,n\in\mathbb{N}$.
For $j\ge0$ we have
\begin{eqnarray*}
\lambda^{j}\frac{d^{j}\widetilde{m}_{k,n}(\lambda)}{d\lambda^{j}} & = & \int\lambda^{j}\frac{\partial^{j}}{\partial\lambda^{j}}p(x,\lambda)\lambda_{k}^{n}\phi_{k}(x)d\mu(x)\\
 & = & \int\lambda^{j}\frac{\partial^{j}}{\partial\lambda^{j}}p(x,\lambda)(-\Delta)^{n}\phi_{k}(x)d\mu(x)\\
 & = & \int\lambda^{j}(-\Delta)^{n}\frac{\partial^{j}}{\partial\lambda^{j}}p(x,\lambda)\phi_{k}(x)d\mu(x),
\end{eqnarray*}
using the Gauss-Green formula. The estimates \eqref{eq:def_symbol_varcoef}
imply that
\begin{equation}
\left\vert \lambda^{j}\frac{d^{j}\widetilde{m}_{k,n}(\lambda)}{d\lambda^{j}}\right\vert \le C_{j,n}\int\vert\phi_{k}(x)\vert d\mu(x)\le C_{j,n}\mu(X)^{1/2}.\label{eq:m_tilde_is_symbol}
\end{equation}
Thus $\widetilde{m}_{k,n}(\lambda)$ is a $0$-symbol. Notice that
the constants in \eqref{eq:m_tilde_is_symbol} are \emph{independent}
of $k$. Now Theorem 4.5.4 of \cite{Kig_CUP01} (and the comments following
it) implies that there is a positive constant $c$ and a real number
$\alpha$ such that
\begin{equation}
\Vert\phi_{k}\Vert_{\infty}\le c\lambda_{k}^{\alpha}\;\mbox{for all}\; k\ge0.\label{eq:boundon_phik}
\end{equation}
Let $n\ge0$ be such that the series $S:=\sum_{k}\lambda_{k}^{\alpha-n}$
converges. For $u\in\D$ we can write then
\begin{eqnarray*}
Tu(x) & = & \sum_{j}\sum_{k}m_{k}(\lambda_{j})\phi_{k}(x)P_{\lambda_{j}}u(x)\\
 & = & \sum_{k}\frac{\phi_{k}(x)}{\lambda_{k}^{n}}\sum_{j}\widetilde{m}_{k,n}(\lambda_{j})P_{\lambda_{j}}u(x)\\
 & = & \sum_{k}\frac{\phi_{k}(x)}{\lambda_{k}^{n}}\widetilde{m}_{k,n}(-\Delta)u(x),
\end{eqnarray*}
where $\widetilde{m}_{k,n}(-\Delta)$ is the pseudo-differential operator
of order $0$ attached to $\widetilde{m}_{k,n}$. Theorem \ref{thm:smooth_kernel}
and estimates \eqref{eq:m_tilde_is_symbol} imply that $\widetilde{m}_{k,n}(-\Delta)$
is given by integration with respect to a kernel $K_{k,n}(x,y)$ that
is continuous (and smooth) off diagonal and satisfies the estimate
\begin{equation}
\vert K_{k,n}(x,y)\vert\le C\cdot R(x,y)^{-d},\label{eq:K_kn}
\end{equation}
where $C$ is a positive constant that is independent of $k$. Then
\[
Tu(x)=\int\sum_{k}\frac{\phi_{k}(x)}{\lambda_{k}^{n}}K_{k,n}(x,y)u(y)d\mu(y).
\]
Thus it suffices to prove that
\[
K(x,y)=\sum_{k}\frac{\phi_{k}(x)}{\lambda_{k}^{n}}K_{k,n}(x,y)
\]
is well defined and continuous. Estimates \eqref{eq:K_kn} imply that
\[
\vert K(x,y)\vert\le C\cdot R(x,y)^{-d}\sum_{k}\frac{\lambda_{k}^{\alpha}}{\lambda_{k}^{n}}=CS\cdot R(x,y)^{-d}.
\]
Thus $K(x,y)$ is continuous off diagonal and \eqref{eq:est_ker_var_coeff}
holds.
\end{proof}
We conjecture that the kernel $K$ of a pseudo-differential operator
$p(x,-\Delta)$ of order $0$ is a Calder\'on-Zygmund operator in
the sense of \cite[Section I.6.5]{Ste93}. Even though we are unable
at this point to prove this claim, we can, nevertheless, show that
these operators are bounded on $L^{q}(\mu)$, for all $1<q<\infty$.
\begin{thm}
Suppose that $X$ is a compact fractafold with no boundary and $p\in S^{0}$.
Then the operator $Tu(x)=p(x,-\Delta)u(x)$ extends to a bounded operator
on $L^{q}(\mu)$ for all $1<q<\infty$.\end{thm}
\begin{proof}
Fix $1<q<\infty$ and let $r$ such that $1/q+1/r=1$. Let $u\in\D$
and let $\widetilde{m}_{k,n}$ as in the proof of Theorem \ref{thm:continuous_ker}.
Recall that $\widetilde{m}_{k,n}$ is a symbol of order $0$ with
constant coefficients, and
\[
Tu(x)=\sum_{k}\frac{\phi_{k}(x)}{\lambda_{k}^{n}}\widetilde{m}_{k,n}(-\Delta)u(x),
\]
for all $n\ge0$. Recall from inequality \eqref{eq:boundon_phik} that
there is $\alpha$ such that $\Vert\phi_{k}\Vert_{\infty}\le c\lambda_{k}^{\alpha}$
for some positive constant $c$. Let $n\ge0$ such that the series
$S_{1}=\sum_{k}\lambda_{k}^{r(\alpha-n/2)}$ and $S_{2}=\sum_{k}\lambda_{k}^{-nq/2}$
are convergent. Estimates  (\ref{eq:m_tilde_is_symbol}) and Corollary
\ref{pro:Prop_boundedness} imply that  $\widetilde{m}_{k,n}(-\Delta)$  is a
bounded operator on $L^{q}(\mu)$ with the bound \emph{independent} on $k$. Then, H\"older inequality implies
that
\begin{eqnarray*}
\Vert Tu\Vert_{q}^{q} & = & \int\left\vert \sum_{k}\frac{\phi_{k}(x)}{\lambda_{k}^{n}}\widetilde{m}_{k,n}(-\Delta)u(x)\right\vert ^{q}d\mu(x)\\
 & \le & \int\left(\sum_{k}\left\vert \frac{\phi_{k}(x)}{\lambda_{k}^{n/2}}\frac{\widetilde{m}_{k,n}(-\Delta)u(x)}{\lambda_{k}^{n/2}}\right\vert \right)^{q}d\mu(x)\\
 & \le & \int\left(\sum_{k}\frac{\vert\phi_{k}(x)\vert^{r}}{\lambda_{k}^{nr/2}}\right)^{\frac{q}{r}}\cdot\left(\sum_{k}\frac{\vert\widetilde{m}_{k,n}(-\Delta)u(x)\vert^{q}}{\lambda_{k}^{nq/2}}\right)d\mu(x)\\
 & \le & c^{q}\int\left(\sum_{k}\lambda_{k}^{r(\alpha-n/2)}\right)^{\frac{q}{r}}\sum_{k}\frac{\vert\widetilde{m}_{k,n}(-\Delta)u(x)\vert^{q}}{\lambda_{k}^{nq/2}}d\mu(x)\\
 & = & c^{q}S_{1}\sum_{k}\frac{1}{\lambda_{k}^{nq/2}}\int\vert\widetilde{m}_{k,n}(-\Delta)u(x)\vert^{q}d\mu(x)\\
 & \le & C(n)S_{1}S_{2}\Vert u\Vert_{q}^{q}.
\end{eqnarray*}
Thus $T$ extends to a bounded operator on $L^{q}(\mu)$.
\end{proof}

\bibliographystyle{amsplain}
\bibliography{bibliography}

\end{document}